\documentclass[reqno,11pt]{amsart}

\usepackage{palatino}

\usepackage{amsmath}
\usepackage{amsfonts,amssymb,amsmath,latexsym,wasysym,mathrsfs,bbm,stmaryrd}
\usepackage[all]{xy}
\usepackage[pdftex]{graphicx,color}

\usepackage[mathcal]{euscript}

\usepackage{graphics,moreverb,hangcaption,gastex}

\def\g{\gamma}
\def\a{\alpha}
\def\l{\lambda}

\def\R{\mathbb R}
\def\C{\mathbb C}

\def\N{\mathbb N}

\def\1{\mathbbm 1}

\def\l{\lambda}

\def\tensor{\otimes}

\def\ta{\tilde{a}}
\def\td{\tilde{d}}

\newcommand{\interior}[1]{\raise0.2ex\hbox{$\displaystyle{\mathop{#1}^{\circ}}$}}
\newcommand{\mx}[1]{\mathbf{#1}}

\def\l{\lambda}

\newtheorem{theorem}{Theorem}[section]
\newtheorem*{theorem*}{Theorem}

\newtheorem*{theorem-law}{Theorem \ref{thm 4th moment semicircle}}
\newtheorem*{theorem-cont}{Theorem \ref{thm 4th moment cont}}
\newtheorem*{theorem-transfer}{Theorem \ref{thm transfer}}
\newtheorem*{corollary-nosemicirc}{Corollary \ref{cor no semicircular}}
\newtheorem*{theorem-Malliavin}{Theorem \ref{theorem Malliavin estimate}}
\newtheorem*{corollary-Wasserstein}{Corollary \ref{cor Wasserstein}}

\newtheorem{proposition}[theorem]{Proposition}
\newtheorem{definition}[theorem]{Definition}
\newtheorem{corollary}[theorem]{Corollary}

\newtheorem{lemma}[theorem]{Lemma}

\numberwithin{equation}{section}

\theoremstyle{remark}
\newtheorem*{remark*}{Remark}
\newtheorem{remark}[theorem]{Remark}

\theoremstyle{remark}

\long\def\symbolfootnote[#1]#2{\begingroup%
\def\thefootnote{\fnsymbol{footnote}}\footnote[#1]{#2}\endgroup}

\begin{document}

\title{Duality in Segal-Bargmann Spaces}
\author{William E. Gryc$^{(1)}$}
\address{$(1)$ Muhlenberg College, Allentown, PA}
\email{wgryc@muhlenberg.edu}
\author{Todd Kemp$^{(2)}$}
\address{$(2)$ UCSD, La Jolla, CA}
\email{tkemp@math.ucsd.edu}

\begin{abstract} For $\a>0$, the {\em Bargmann projection} $P_\alpha$ is the orthogonal projection from $L^2(\g_\a)$ onto the holomorphic subspace $L^2_{hol}(\g_\a)$, where $\g_\a$ is the standard Gaussian probability measure on $\C^n$ with variance $(2\a)^{-n}$.  The space $L^2_{hol}(\gamma_\a)$ is classically known as the {\em Segal-Bargmann space}.  We show that $P_\a$ extends to a bounded operator on $L^p(\gamma_{\a p/2})$, and calculate the exact norm of this scaled $L^p$ Bargmann projection.  
We use this to show that the dual space of the $L^p$-Segal-Bargmann space $L^p_{hol}(\g_{\a p/2})$ is an $L^{p'}$ Segal-Bargmann space, but with the Gaussian measure scaled differently: $(L^p_{hol}(\g_{\a p/2}))^{\ast} \cong L^{p'}_{hol}(\g_{\a p'/2})$ (this was shown originally by Janson, Peetre, and Rochberg).  We show that the Bargmann projection controls this dual isomorphism, and gives a dimension-independent estimate on one of the two constants of equivalence of the norms.
\end{abstract}


\maketitle

\symbolfootnote[0]{(2) Supported in part by the NSF Grants DMS-0701162 and DMS-1001894.}

\section{Introduction and Background}

The Fock space is a central object in quantum mechanics, operator algebras, and probability theory.  Based over the Hilbert space $\C^n$, it can be identified as a Hilbert space of holomorphic functions.  Let $\alpha>0$ and let $\g_\a=\g_\a^n$ denote the Gaussian measure
\begin{equation} \label{eq g_a} \g_\a(dz) = \left(\frac{\a}{\pi}\right)^n e^{-\a|z|^2}\,\l(dz), \end{equation}
where $\l$ is the Lebesgue measure on $\C^n$.  Then the Fock space is $\mathscr{F}_\a=\mathscr{F}_\a(\C^n) \equiv L^2_{hol}(\C^n,\g_\a)$, the (entire) holomorphic functions in $L^2(\C^n,\g_\a)$.  It is a reproducing kernel Hilbert space, with kernel
\begin{equation} \label{eq K_a} \mathcal{K}_\a(z,w) = e^{\alpha\langle z,w\rangle}. \end{equation}
(Note: $\langle z,w\rangle$ denotes the complex inner-product $\sum_{i=1}^n z_i \overline{w_i}$.)  As usual, the existence of the reproducing kernel guarantees that $\mathscr{F}_\a$ is, in fact, a closed subspace of $L^2(\g_\a)$.

In this paper, we study the orthogonal projection $P_\a\colon L^2(\C^n,\g_\a) \to L^2_{hol}(\C^n,\g_\a)=\mathscr{F}_\a$.
As in any reproducing kernel Hilbert subspace, this orthogonal projection has the reproducing kernel itself as its integral kernel,
\begin{equation} \label{eq P_a} P_\a f(z) = \int_{\C^n} e^{\alpha\langle z,w\rangle}f(w)\,\g_\alpha(dw). \end{equation}
The projection $P_\a$ is the $(\C^n,\g_\a)$ equivalent of the classical {\em Bergman projection} (in Bergman spaces on the unit disk in $\C$); in more general contexts it is sometimes called the {\em Riesz projection}.  Since its range is the classical Segal-Bargmann space, we refer to $P_\a$ as the {\em Bargmann projection}.  (Note, it is not the same object as the Segal-Bargmann transform cf.\ \cite{Bargmann,Segal}, though there are obvious connections.)   $P_\a$ naturally controls the geometry of the imbedding of $\mathscr{F}_\alpha$ into $L^2(\g_\a)$; the interpolation scale of these holomorphic spaces can be understood well in its context.  The unusual duality properties of the holomorphic $L^p$-spaces of $\g_\a$ (as discussed in Section \ref{section duality} below) have the result that $P_\alpha$ is not bounded on $L^p(\g_\a)$ for $p\neq 2$; rather, the measure must be scaled.  The main result of this paper is the following theorem.

\begin{theorem} \label{main theorem 1} Let $n$ be a positive integer, let $1\le p<\infty$, and let $\alpha>0$.  Let $p'$ denote the conjugate exponent to $p$, $\frac{1}{p}+\frac{1}{p'}=1$.  The Bargmann projection $P_\a$ is bounded on $L^p(\C^n,\g_{\a p/2})$, with norm
\begin{equation} \label{eq main theorem 1} \|P_\alpha\colon L^p(\C^n,\g_{\a p/2}) \to L^p_{hol}(\C^n,\g_{\a p/2})\| =\left(2\frac{1}{p^{1/p}}\frac{1}{p'^{1/p'}}\right)^n.
\end{equation}
\end{theorem}
When $p=2$, the norm in Equation \eqref{eq main theorem 1} is equal to $1$ in all dimensions, as expected for an orthogonal projection; for all other $p$, it grows exponentially with dimension.  In particular, the $L^1(\gamma_{\alpha/2})$-norm of $P_\alpha$ is $2^n$.  Note that the main theorem of \cite{DostanicZhu} is the upper bound $2^n$ for the norm in Equation \eqref{eq main theorem 1} (a result which is actually contained in \cite{JPR} in a wider context); as we show in Section \ref{section holomorphic sections}, this upper bound follows simply from a reinterpretation of $\mathscr{F}_\alpha$ as a subspace of $L^p$ functions over Lesbesgue measure.

\medskip

As we discuss in Section \ref{section duality 2}, the norm of $P_\a$ controls the norm of the dual space $\left(L^p_{hol}(\g_{\a p/2})\right)^\ast$.

\begin{theorem} \label{main corollary} Let $1<p<\infty$ and $\a>0$.  Let $p'$ be the conjugate exponent to $p$. In the pairing $(f,g)_\alpha = \int_{\C^n} f\overline{g}\,d\g_\a$, the spaces $L^p_{hol}(\g_{\a p/2})$ and $L^{p'}_{hol}(\g_{\a p'/2})$ are dual.  The norms satisfy
\begin{equation} \label{eq main corollary 2} 
 \|h\|_{L^{p'}_{hol}(\g_{\a p'/2})} \le  \|(\,\cdot\,,h)_\a\|_{(L^p_{hol}(\g_{\a p/2}))^\ast} \le \left(2\frac{1}{p^{1/p}}\frac{1}{p'^{1/p'}}\right)^n \|h\|_{L^{p'}_{hol}(\g_{\a p'/2})}.
\end{equation}
\end{theorem}

\begin{remark} In fact, it is the first inequality in \ref{eq main corollary 2} that is the interesting new result; the second inequality is actually just H\"older's inequality when reinterpreted in terms of $L^p$ spaces over Lesbesgue measure, as explained in Section \ref{section duality 2} below.  \end{remark}

\begin{remark} The authors find it particularly worthy of note that the first inequality in \ref{eq main corollary 2} is {\em independent of  dimension}. \end{remark}

In Section \ref{section holomorphic sections}, we show how the problem may be simplified by viewing elements of the Fock space as elements of a subspace of $L^2$  over Lesbesgue measure; this transformation offers a new explanation for why $P_\alpha$ acts naturally on $L^p(\g_{\alpha p/2})$ rather than $L^p(\g_\a)$ (and hence why the holomorphic $L^p$-spaces of $\g_\a$ do not satisfy the usual duality relations).  Since $P_\alpha$, given by Equation \eqref{eq P_a}, has a Gaussian kernel, our approach is to use the results of \cite{Lieb} to calculate the norm which occurs on the subspace of Gaussian functions.  Since the kernel of $P_\a$ is complex, the Gaussian maximizer may also be complex, which greatly complicates the computations.

\medskip

The remainder of this paper is organized as follows.  Section \ref{section duality} explores the unusual duality relations among the holomorphic $L^p$ spaces of Gaussian measures.  In Section \ref{section holomorphic sections}, we reinterpret $L^p_{hol}(\gamma_{\a p/2})$ as a subspace of $L^p$ over Lesbesgue measure, which sheds light on the rescaling required for the usual $L^p$-duality.  This allows us to reinterpret the projection $P_\alpha$ as a new operator $Q_\alpha$ in the setting of Lebesgue measure, where it is easier to analyze.  In Section \ref{section elementary bounds}, we show that the norm of $P_\alpha$ on $L^p$ grows exponentially with dimension, and in Section \ref{section duality 2},  we use the Lesbesgue perspective to prove Theorem \ref{main corollary}.  Section \ref{section Proof of Main Theorem} is devoted to the proof of Theorem \ref{main theorem 1}.  In Section \ref{sect Segal}, we reduce the calculation to the $n=1$ dimensional case with a version of Segal's lemma for tensor products of integral operators.  A deep result of Lieb, cf.\ \cite{Lieb}, is then used in Section \ref{section Gaussian kernels} to further reduce to the case of putative Gaussian maximizers for the norm of $P_\a$.  Sections \ref{section formula for ratio} and \ref{section Optimization} then setup the appropriate calculus problem to determine the norm.  The proof is completed with the lengthy calculations of Sections \ref{section critical point} and \ref{section maximum at critical point}, determining critical points and identifying the global maximum to calculate the sharp norm of $P_\a$, concluding the paper.

\subsection{Gaussian Integrals} \label{section Gaussian integrals} Many of the calculations throughout this paper rely on the following formula for integrating Gaussian functions.  Let $A$ be a $k\times k$ complex symmetric matrix, whose real part $\Re A$ is positive definite.  Let $v\in\C^k$, and let $(\cdot,\cdot)$ denote the {\em real} inner-product extended (bilinearly, not sesquilinearly) to $\C^k$.  Then the (uncentered) Gaussian function $x\mapsto e^{-(x,Ax)+2(v,x)}$ is in $L^1(\R^k)$, and
\begin{equation}\label{gaussianintegral}
\int_{\mathbb{R}^k}e^{-(x,Ax)+2(v,x)}dx =\frac{\pi^{k/2}}{\sqrt{\det(A)}}e^{(v,A^{-1}v)}.
\end{equation}
Equation \eqref{gaussianintegral} can be found as \cite[Ex. 5, Ch. 5]{LiebLoss}.  It is easy to verify for real $A$ (by diagonalizing and completing the square); the general formula then follows by an analytic continuation argument. 

\subsection{Duality in $L^p_{hol}(\g_\a)$} \label{section duality} In the holomorphic space $\mathscr{F}_\alpha(\C^n) = L^2_{hol}(\C^n,\g_\a)$, Taylor series expansions are, in fact, orthogonal sums (since the measure $\g_\a$ is rotationally-invariant).  Indeed, it is easy to compute that the monomials $z^\mx{j} = z_1^{j_1}\cdots z_n^{j_n}$ are orthogonal with
\begin{equation} \label{eq orth rel} ( z^\mx{j},z^\mx{k} )_\a  = \delta_{\mx{j}\mx{k}}\frac{\mx{j}!}{\alpha^{|\mx{j}|}}, \end{equation}
where $\mx{j}! = j_1!\cdots j_n!$ and $|\mx{j}| = j_1+\cdots+j_n$.  Letting $\phi_\mx{j}(z) = (\alpha^{|\mx{j}|/2}/\sqrt{\mx{j}!})z^{\mx{j}}$,  Taylor's theorem therefore asserts that $\{\phi_\mx{j}\,:\,\mx{j}\in \N^n\}$ is an orthonormal basis for $L^2_{hol}(\C^n,\g_\a)$.  This justifies the claim that the reproducing kernel is given as in Equation \eqref{eq K_a}, since
\[ \mathcal{K}_\alpha(z,w) \equiv \sum_{\mx{j}\in\N^n} \phi_\mx{j}(z)\overline{\phi_{\mx{j}}(w)} = \sum_{\mx{j}\in\N^n} \frac{\alpha^{|\mx{j}|}}{\mx{j}!} z^{\mx{j}}\overline{w}^{\mx{j}} = e^{\alpha\langle z,w\rangle}. \]
The standard estimate $|f(z)|\le \mathcal{K}_\alpha(z,z)\|f\|_2$ shows that $L^2(\g_\a)$-convergence implies pointwise convergence in $L^2_{hol}(\g_\a)$; from this it is easy to see that $L^2_{hol}(\g_\a)$ is a Hilbert space.

\medskip

For $p\ne 2$, the spaces $L^p_{hol}(\g_\a)$ behave somewhat differently than one would na\"ively expect.  Let $1<p<\infty$, and let $p'$ be its conjugate exponent.  If $h\in L^{p'}_{hol}(\g_\a)$ then $h\in L^{p'}(\g_\a)$ and so can be viewed, via the usual pairing, as an element $(\,\cdot\,,h)_\a$ of the dual space to $L^p(\g_\a)$:
\[ L^p(\g_\a)\ni f\mapsto ( f, h )_\alpha = \int f\overline{h}\,d\g_\a. \]
Naturally this imbedding does not give all of $(L^p(\g_\a))^\ast$ since $L^{p'}_{hol}(\g_\a)$ is a small subspace of $L^{p'}(\g_\a)$.  The same imbedding shows that $(\,\cdot\,,h)_\a$ is an element of the dual space to $L^p_{hol}(\g_\a)$; it is somewhat surprising that, in this context as well, the set of all $(\,\cdot\,,h)_\a$ with $h\in L^{p'}_{hol}(\g_\a)$ is {\em not} the full dual space (unless $p=2$).  This was discovered by Sj\"ogren, cf.\ \cite{Sjogren}; for completeness, we reproduce the following simpler argument, which is due to Carlen and Gross.

\begin{proposition} \label{prop not dual} If $1<p<\infty$ and $p\ne 2$, then the imbedding $L^{p'}_{hol}(\g_\a)\to (L^p_{hol}(\g_\a))^\ast$ is not surjective. 
\end{proposition}

\begin{proof} First note that the map $h\mapsto (\,\cdot\,,h)_\a$ is injective.  For if $(f,h)_\a =0$ for all $f\in L^{p'}_{hol}(\g_\a)$, we may take $f$ to be a monomial $f(z) = z^{\mx{j}}$ (which is in $L^{p'}_{hol}(\g_\a)$ for all $p>1$); the orthogonality relations of Equation \eqref{eq orth rel} then yield $(f,h)_\a = \mx{j}!/\alpha^{|\mx{j}|}\cdot T_{\mx{j}}(h)$ where $T_{\mx{j}}(h)$ is the $\mx{j}$th Taylor-coefficient of the holomorphic function $h$.  Hence, the Taylor series of $h$ is $0$, and so $h=0$.

\bigskip

From H\"older's inequality, the map $h\mapsto (\,\cdot\,,h)_\a$ is, as usual, bounded.  Hence, by the Open Mapping Theorem, if it is also surjective it follows that it has a bounded inverse.  We will show this is not true by demonstrating there is no constant $C>0$ such that
\begin{equation}\label{duality constant}
\|h\|_{p'} \le C\|(\,\cdot\,,h)_\a\|_{(L_{hol}^p(\g_\a))^\ast}, \quad \text{for all $h\in L_{hol}^{p'}(\g_\a)$}.
\end{equation}
Indeed, consider the function $h(z)=K_w(z) = \mathcal{K}_\alpha(z,w) = e^{\alpha\langle z,w\rangle}$, which is of course in $L_{hol}^{p'}(\g_\a)$.  A simple computation using Equation \eqref{gaussianintegral} shows that
\begin{equation}\label{p' norm of Kw}
\|K_w\|_{p'} = e^{p'\alpha|w|^2/4}.
\end{equation}
Now, $K_w$ is the reproducing kernel; that is, $(f,K_w)_\alpha = f(w) \equiv \Lambda_w(f)$ for all $f\in L^2_{hol}(\g_\a)$, and so therefore also for $f\in L^p_{hol}(\g_\a)$ (true for $p>2$ since $\g_\a$ is a finite measure so $L^p(\g_\a)\subseteq L^2(\g_\a)$; true for $p<2$ since $L^p_{hol}(\g_\a)$ is dense in $L^2_{hol}(\g_\a)$ and $(\,\cdot\,,K_w)_\a$ is continuous on $L^p(\g_\a)$).  Hence, the norm on the right-hand side of Equation (\ref{duality constant}) is just the $(L^p_{hol}(\g_\a))^\ast$-norm of the evaluation functional $\Lambda_w$.  This
is computed elegantly by Carlen in \cite[Thm. 3]{Carlen}; the result is
\begin{equation}\label{norm of eval}
\|\Lambda_w\|_{(L_{hol}^p(\g_\a))^\ast} = e^{\alpha|w|^2/p}.
\end{equation}
Therefore, Equation (\ref{duality constant}) implies that
\[ e^{p'\alpha|w|^2/4} \le C e^{\alpha|w|^2/p},\text{ for all $w\in\C$.} \]
Rearranging and simplifying, we find
\begin{equation}\label{comp}
C \ge e^{\alpha|w|^2(p'/4-1/p)} = e^{\alpha|w|^2(p-2)^2/4p(p-1)}.
\end{equation}
Since $(p-2)^2/4p(p-1)>0$ as long as $p\ne 2$, the right-hand side of Equation (\ref{comp})
tends to $\infty$ as $|w|\to\infty$.  Hence, there can be no such constant $C$.
\end{proof}

This surprising lack of duality has material consequences for the Bargmann projection $P_\a$.

\begin{corollary} \label{cor P_a not bounded} The Bargmann projection $P_\a$ is not bounded on $L^p(\g_\a)$ for any $p\ne 2$. \end{corollary}

\begin{remark} $P_\a$ acts, by definition, on $L^2(\g_\a)$, and so for $p>2$ the action of $P_\a$ on $L^p(\g_\a)$ is well-defined.  For $p<2$, the corollary should be interpreted to say that $P_\a$ is not bounded on $L^2(\g_\a)\cap L^p(\g_\a)$, and hence has no extension to $L^p(\g_\a)$.
\end{remark}

\begin{remark} The idea of this proof is due to Brian Hall. \end{remark}

\begin{proof} Suppose, to the contrary, that, for some $p\ne 2$, $P_\alpha$ is bounded from $L^p(\g_\a)$ to $L_{hol}^p(\g_\a)$.  (If $p<2$, the supposition is that $\left.P_\a\right|_{L^2\cap L^p}$ extends continuously to $L^p$.)  It then
follows (by the self-adjointness of $P_\a$ on $L^2$ and H\"older's inequality) that $P_\a$ is also
bounded from $L^{p'}(\g_\a)$ to $L_{hol}^{p'}(\g_\a)$.

Let $\Phi$ be any linear functional in $(L^p_{hol}(\g_\a))^\ast$.  We may then define a linear
functional $\hat{\Phi}\in(L^p(\g_\a))^\ast$ by
\begin{equation}\label{Phi hat}
\hat{\Phi}(f) = \Phi(P_\a f).
\end{equation}
(Note, from Equation \eqref{eq P_a}, $P_\a$ is related to the Fourier transform, and so it makes sense to so-name the new functional $\hat{\Phi}$.) Since $P_\a$ is a projection, $P_\a^2 = P_\a$, and so $\hat\Phi(P_\a f) = \Phi(P_\a^2f) = \Phi(P_\a f) = \hat\Phi(f)$.  Since $P_\a$ is bounded on $L^p(\g_\a)$, $\hat\Phi\in(L^p(\g_\a))^\ast = L^{p'}(\g_\a)$.  That is, there is a unique function $g\in L^{p'}(\g_\a)$ such that $\hat\Phi(f) = (f,g)_\a$.  Note, then, that
$\hat\Phi(f) = \hat\Phi(P_\a f) = (P_\a f,g)_\a = (f,P_\a g)_\a$.  (The last equality holds if $f,g\in L^2(\g_\a)$, and so
holds in general by the denseness of $L^p$ in $L^2$.)  But then $(f,g)_\a=(f,P_\a g)_\a$ for all $f\in L^p$.
Since $g\in L^{p'}$ and so $P_\a g\in L^{p'}$ (by the absurd assumption of the proof), it follows that
$P_\a g = g$, and so $g\in L^{p'}_{hol}(\g_\a)$.

Thus, the map $L^{p'}_{hol}(\g_\a)\to(L_{hol}^p(\g_\alpha))^\ast$ which sends $g$ to the linear functional
$(\,\cdot\,,g)_\a$ is surjective and continuous.  This contradicts Proposition \ref{prop not dual}.
\end{proof}

\begin{remark} Corollary \ref{cor P_a not bounded} shows that there is a close connection between projections in $L^2$ and the duality relations in closed subspaces of $L^p$.  We could have proved this result in great generality, but it is only relevent for us in this limited context. \end{remark}

In \cite{JPR}, the authors identify (up to scale) what the actual dual space of $L^p_{hol}(\g_\a)$ is, by reinterpreting the action of the Bargmann projection $P_\a$.  Their results apply to a much more general setting than the Gaussian measures $\g_\a$.  If $\mu$ is a measure on a connected region $\Omega\subseteq \C^n$, possessing a strictly-positive density, and if the group of gauge transformations (holomorphic bijections $\upsilon$ of $\Omega$ with the property that $(\upsilon^{-1})_\ast\mu = |\phi|^2\mu$ for some holomorphic gauge factor $\phi$) is sufficiently rich, then the orthogonal projection $P\colon L^2(\mu)\to L^2_{hol}(\mu)$ should really be thought of as a map from $L^2[\mathcal{K}]\to L^2_{hol}[\mathcal{K}]$.  Here $\mathcal{K}$ is the reproducing kernel of $L^2_{hol}(\mu)$, and $L^p[\mathcal{K}]$ is a {\em weighted} $L^p$-space, defined as the set of all functions $f$ such that $f(z)/\mathcal{K}(z,z)^{1/2}$ is in $L^p(\mathcal{K}(z,z)\,\mu(dz))$ (with the natural norm).  Of course $L^2[\mathcal{K}] = L^2(\mu)$, but for $p\ne 2$ they are distinct.  Janson, Peetre, and Rochberg show that $P$ extends to a bounded map from $L^p[\mathcal{K}]\to L^p_{hol}[\mathcal{K}]$ (with norm $\le 2^n$) for such sufficiently nice $\mu$.  This leads to the correct identification of the dual space to $L^p_{hol}(\g_\a)$.

\subsection{The Lesbesgue Setting and the Operator $Q_\alpha$} \label{section holomorphic sections}

Following the discussion at the end of Section \ref{section duality}, and noting that $\mathcal{K}_\a(z,z)^{1/2} = e^{\frac{\alpha}{2}|z|^2}$, we should consider the following spaces.
\begin{definition} \label{def section} For $\alpha>0$, let $\mathcal{S}_\alpha$ denote the space 
\begin{equation} \mathcal{S}_\a = \{F(z) = f(z)e^{-\frac{\a}{2}|z|^2}\,:\,f\text{ is holomorphic on }\C^n\}. \end{equation}
For $1\le p<\infty$, define $\mathcal{S}^p_\a = \mathcal{S}_\a\cap L^p(\C^n,\l)$.
\end{definition}
Consider the multiplier map $\mathfrak{g}_\a$
\begin{equation} \label{eq gg_a} (\mathfrak{g}_\a f)(z) = e^{-\frac{\a}{2}|z|^2}f(z), \end{equation}
determined by the density of the measure $\g_{\a/2}$.  Thus $\mathcal{S}_\a = \mathfrak{g}_\a \mathrm{Hol}(\C^n)$.   The norm on $\mathcal{S}_\a^p$ is given by Lebesgue measure.  It is easy to see that $\mathcal{S}_\a^p$ is a closed subspace of $L^p(\C^n,\l)$.  In particular, there is an orthogonal projection $Q_\a$
\begin{equation} Q_\alpha\colon L^2(\C^n,\l)\to \mathcal{S}^2_\a. \end{equation}
$Q_\a$ is a reinterpretation of $P_\a$, as we now explain.  We can use the map $\mathfrak{g}_\a$ to connect the $\mathcal{S}_\a^p$ spaces with the $L_{hol}^p(\g_\a)$ spaces.
Indeed, $\mathfrak{g}_\alpha f\in \mathcal{S}_\a$ iff $f$ is holomorphic. A simple calculation reveals that $\mathfrak{g}_\a$ is a dilation from $L^p(\l)$ to $L^p(\g_{\a p/2})$.
\[ \begin{aligned}
\|\mathfrak{g}_\a f\|^p_{L^p(\l)} & = & \int_{\C^n} \left|f(z)e^{-\frac{\a}{2}|z|^2}\right|^p\,\l(dz) \\
                           & = & \int_{\C^n} |f(z)|^p e^{-\frac{\a p}{2} |z|^2}\,\l(dz) \\
& = & \left(\frac{2\pi}{p\alpha}\right)^n \| f\|^p_{L^p(\g_{\a p/2})}.\label{image of g}
\end{aligned} \]
The multiplier function is strictly positive, and so $\mathfrak{g}_\a$ is a bijection.  Hence, rescaling the multiplication map
\begin{equation} \label{eq g_ap} \mathfrak{g}_{\a,p} = \left(\frac{p\a}{2\pi}\right)^{n/p}\mathfrak{g}_\a, \end{equation}
we have the following.

\begin{proposition} \label{prop commutative diagram} Let $1\le p<\infty$ and $\a>0$. The map $\mathfrak{g}_{\a,p}$ of Equations \eqref{eq gg_a} and \eqref{eq g_ap} is an isometric isomorphism $L^p(\g_{\a p/2})\to L^p(\l)$. Its restriction $\mathfrak{g}_{\a,p}\colon L^p_{hol}(\g_{\a p/2})\to \mathcal{S}_\a^p$ is also an isometric isomorphism.  Hence, the following diagram commutes.
\[
\xymatrix{
L^p(\g_\a) \ar[r]^{\mathfrak{g}_{\alpha,p}} \ar[d]_{P_\alpha} & L^p(\l) \ar[d]^{Q_\alpha} \\
L_{hol}^p(\g_\a) \ar[r]_{\mathfrak{g}_{\alpha,p}} & \mathcal{S}^p_\a
}
\]
\end{proposition}

\begin{remark} One may simply use the map $\mathfrak{g}_\a$ in place of $\mathfrak{g}_{\alpha,p}$ in the diagram, but we
find this setup more \ae sthetically pleasing; here, the horizontal arrows are isometric isomorphisms, and the vertical arrows are orthogonal projections. \end{remark}

Thus $Q_\a = \mathfrak{g}_{\a,p} P_\a \mathfrak{g}_{\a,p}^{-1} = \mathfrak{g}_\a P_\a \mathfrak{g}_\a^{-1}$ is the conjugated action of the Bargmann projection, from the standard $L^2$ space $L^2(\C^n,\l)$ onto $\mathcal{S}_\a^2$.  From Equation \eqref{eq P_a}, this means $Q_\a$ has the integral representation

\begin{equation} \label{eq Q_a} \begin{aligned}
Q_\a F(z) &= \mathfrak{g}_\a\left(\int_{\C^n} e^{\alpha\langle z,w\rangle}(\mathfrak{g}_\a^{-1}F)(w)\,\g_\a(dw)\right) \\
&= \left(\frac{\a}{\pi}\right)^n \int_{\C^n} e^{-\frac{\a}{2}|z|^2+\a\langle z,w\rangle -\frac{\a}{2}|w|^2}F(w)\,\l(dw) = \int_{\C^n} \mathcal{Q}_\a(z,w) F(w)\,\l(dw). 
\end{aligned}
\end{equation}
Here $\mathcal{Q}_\a(z,w) = (\frac{\a}{\pi})^n e^{-\frac{\a}{2}(|z|^2 - 2\langle z,w \rangle + |w|^2)}$ is the kernel of $Q_\a$.  In this form, $Q_\alpha$ may (a priori) act on $L^p(\l)$ for any $p$.  To see that it does so boundedly, consider the operator $|Q_\a|$ whose integral kernel is $|\mathcal{Q}_\a|$:
\begin{align}\nonumber |Q_\a|F(z) &= \left(\frac{\a}{\pi}\right)^n \int_{\C^n} |e^{-\frac{\a}{2}|z|^2+\a\langle z,w\rangle -\frac{\a}{2}|w|^2}| F(w)\,\l(dw) \\
\label{eq |Qa|} &= \left(\frac{\a}{\pi}\right)^n \int_{\C^n} e^{-\frac{\a}{2}|z-w|^2} F(w)\,\l(dw) = 2^n\cdot (\g_{\a/2}\ast F) (w).
\end{align}
That is, $|Q_\a|$ is convolution with $2^n \g_{\a/2}$. (Here and in the sequel, we let the symbol $\g_\a$ do double duty, representing both the measure and its density.)  Young's convolution inequality therefore provides the following.
\begin{proposition} \label{prop |Q_a|} For $1\le p<\infty$ and $\a>0$,
\[ \| |Q_\a|\colon L^p(\C^n,\l)\to L^p(\C^n,\l) \| = 2^n. \]
\end{proposition}

\begin{remark} Proposition \ref{prop |Q_a|} is actually the main theorem in \cite{DostanicZhu}.  Our proof is different from theirs, and is quite elementary. \end{remark}

\begin{proof} By Young's convolution inequality, $\| |Q_\a|F\|_p = 2^n\|\g_{\a/2}\ast F\|_p \le 2^n\|\g_{\a/2}\|_1 \|F\|_p$, and $\g_\a$ is a probability density so $\|\g_{\a/2}\|_1 = 1$.  To see that the inequality is saturated, take $F = \g_\beta$ for any $\beta>0$, which is in $L^p$ for any $p>0$; since the Gaussian probability measures $\g_\beta$ form a convolution semigroup, it follows that $|Q_\a|(\g_\beta) = 2^n\g_{\a/2}\ast\g_\beta = 2^n \g_{\a/2+\beta}$.  A quick calculation using Equation \eqref{gaussianintegral} shows that
\begin{equation} \label{eq Lp norm of g_b} \|\g_\beta\|_p =  \left(\int_{\R^{2n}} \left[\left(\frac{\beta}{\pi}\right)^n e^{-\beta|z|^2}\right]^p\,dz\right)^{1/p}  = \pi^{(1/p-1)n}p^{-n/p} \beta^{(1-1/p)n}. \end{equation}
Therefore
\[ \frac{\||Q_\a|\g_\beta\|_p}{\|\g_\beta\|_p} = 2^n\frac{\|\g_{\a/2+\beta}\|_p}{\|\g_\beta\|_p} = 2^n\left(\frac{\beta}{\a/2+\beta}\right)^{(1-1/p)n}. \]
For fixed $\a>0$, this tends to $2^n$ as $\beta\to\infty$, concluding the proof.
\end{proof}

\subsection{Elementary Bounds on the norm of $P_\alpha$} \label{section elementary bounds}

Proposition \ref{prop |Q_a|} immediately yields an upper-bound of $2^n$ for the norm of $Q_\alpha$, and therefore of $P_\alpha$, as the next proposition shows.  We also show that the sharp constant (of Theorem \ref{main theorem 1}) is an easy lower-bound.

\begin{proposition} \label{prop elementary bounds} Let $1\le p<\infty $ and $\alpha>0$.  Then the Bargmann projection $P_\alpha\colon L^p(\gamma_{\alpha p/2}) \to L^p_{hol}(\gamma_{\alpha p/2})$ is bounded, with norm
\begin{equation}\label{eqProp1_15}
\left(2\frac{1}{p^{1/p} p'^{1/p'}}\right)^n \le \|P_\alpha\|_{L^p(\gamma_{\alpha p/2})\to L^p_{hol}(\gamma_{\alpha p/2})} \le 2^n.
\end{equation}
In particular, when $p=1$, the norm is equal to $2^n$.
\end{proposition}

\begin{proof} For any $p\ge 1$ and $F\in L^p(\l)$,
\[ \begin{aligned} \|Q_\a F\|_p^p &= \int_{\C^n} \left|\int_{\C^n} \mathcal{Q}_\a(z,w) F(w)\,\l(dw)\right|^p\,\l(dw) \\
&\le \int_{\C^n} \left(\int_{\C^n} |\mathcal{Q}_\a(z,w)| |F(w)|\,\l(dw)\right)^p\,\l(dw) = \| |Q_\a| |F| \|_p^p.
\end{aligned}\]
Thus, for $F\ne 0$, Proposition \ref{prop |Q_a|} shows that
\[ \frac{\|Q_\a F\|_p}{\|F\|_p} \le \frac{\| |Q_\a||F|\|_p}{\| F \|_p} = \frac{\| |Q_\a||F|\|_p}{\| |F| \|_p} \le \| |Q_\a| \|_{L^p(\l)\to L^p(\l)} \le 2^n. \]
From Proposition \ref{prop commutative diagram}, we have $P_\a = \mathfrak{g}_\a^{-1} Q_\a \mathfrak{g}_\a = \mathfrak{g}_{\a,p}^{-1} Q_\a \mathfrak{g}_{\a,p}$ (the last equality following from the fact that $\mathfrak{g}_{\a,p}$ is just a scalar multiple of $\mathfrak{g}_\a$). Since the map $\mathfrak{g}_{\a,p}$ is an isometric isomorphism from $L^p(\g_{\a p/2})$ onto $L^p(\l)$ and its inverse is an isometric isomorphism from $\mathcal{S}_\a^p$ onto $L^p_{hol}(\g_{\a p/2})$, it therefore follows that
\begin{equation} \label{eq norm Pa Qa} \|P_\a\|_{L^p(\g_{\a p/2})\to L^p_{hol}(\g_{\a p/2})} = \|Q_\a\|_{L^p(\l)\to \mathcal{S}^p_\a} =  \|Q_\a\|_{L^p(\l)\to L^p(\l)}. \end{equation}
Therefore $P_\a$ is bounded on $L^p(\g_{\a p/2})$, with norm $\le 2^n$.  For the lower bound, again we test the norm against functions of the form $F=\g_\beta$ for $\beta>0$; so
\begin{equation} \label{eq lower bound 1} \|P_\a\|_{L^p(\g_{\a p/2})\to L^p_{hol}(\g_{\a p/2})} \ge \frac{\|Q_\a \g_\beta\|_p}{\|\g_\beta\|_p}. \end{equation}
Set $g_\beta(w)=e^{-\beta|w|^2}$, so that $\g_\beta = (\beta/\pi)^n g_\beta$.  Then the ratio on the right-hand-side of Equation \eqref{eq lower bound 1} is equal to $\|Q_\a g_\beta\|_p/\|g_\beta\|_p$.  This latter ratio is calculated (as a special case) in Equation \eqref{mainnormquotient} in Section \ref{section formula for ratio}.  (To match up with that formula we take $A = \beta I_{2n}$; thus $A'+I_{2n} = (\frac{2}{\a}\beta+1)I_{2n}$ is a real matrix which commutes with $J$, and thus $\Omega((A'+I_{2n})^{-1}) = 0$.)  Thus
\begin{equation} \label{eq Q_a g_b 1}
\frac{\|Q_\a g_\beta\|_p^p}{\|g_\beta\|_p^p} = 2^{np}\sqrt{\frac{\det(A')}{|\det(A'+I_{2n})|^p}} = 2^{np}\frac{\left(\frac{2}{\a}\beta\right)^n}{|\frac{2}{\a}\beta+1|^{np}} = \left(\frac{2^pc}{(1+c)^p}\right)^n,
\end{equation}
where $c = \frac{2}{\a}\beta>0$.  Elementary calculus shows that, when $p>1$, this is maximized uniquely when $c = \frac{1}{p-1}$, and so (taking $p$th roots) Equations \eqref{eq lower bound 1} and \eqref{eq Q_a g_b 1} yield
\begin{equation} \label{eq lower bound 2}
\|P_\a\|_{L^p(\g_{\a p/2})\to L^p_{hol}(\g_{\a p/2})} \ge \left(\frac{2(\frac{1}{p-1})^{1/p}}{1+\frac{1}{p-1}}\right)^n = \left(2\frac{1}{p^{1/p}}\frac{1}{p'^{1/p'}}\right)^n.
\end{equation}
This proves the proposition for $p>1$.  Note that
\[ \lim_{p\downarrow 1} \frac{1}{p^{1/p}}\frac{1}{p'^{1/p'}} = 1, \]
and so the lower-bound and upper-bound in Equation \eqref{eqProp1_15} converge to $2^n$ as $p\downarrow 1$.  For a precise proof of the $L^1$ lower-bound, we utilize Equation \eqref{eq Q_a g_b 1} in the case $p=1$; thus, for any $\beta>0$,
\[ \frac{\|Q_\a g_\beta\|_1}{\|g_\beta\|_1} = 2^n \left(\frac{c}{1+c}\right)^n \]
where $c=\frac{2}{\alpha}\beta$.  Letting $\beta\to\infty$, so $c\to\infty$, we see the ratio approaches $2^n$, proving the lower-bound; thus $\|Q_\alpha\colon L^1(\l)\to L^1(\l)\| = 2^n$, and the $L^1$-result follows from Equation \eqref{eq norm Pa Qa}.
\end{proof}

\begin{remark} In the preceding proof, we showed that $Q_\alpha$ is bounded on $L^p(\C^n,\l)$ for any $\a>0$.  Through the transformations $\mathfrak{g}_{\a,p}: L^p(\gamma_{\alpha p/2})\to L^p(\lambda)$, this shows why the Bargmann projection $P_\a$ indexed by $\a$ is bounded on the scaled space $L^p(\g_{\a p/2})$, rather than the space $L^p(\g_\a)$ we may have na\"ively expected.  As Section \ref{section duality 2} demonstrates, this is the reason for the unusual scaling properties of the dual spaces of $L^p_{hol}(\g_\a)$.
\end{remark}

\begin{remark} The proof of the lower bound in Proposition \ref{prop elementary bounds} above actually shows that, among Gaussian test functions of the form $g(w) = e^{-\beta|w|^2}$, there is a unique maximizer for the norm of $Q_\a$, which yields the lower bound (when $p>1$).  As we will explain in Section \ref{section Gaussian kernels}, to generalize this technique to determine the sharp norm of $Q_\a$ (which is given by the calculated lower bound in general), we need to expand this maximization only to the class of centered Gaussian functions, of the form $g(x) = e^{-(x,Ax)}$ (now thinking of the variable $x\in\R^{2n}$, with $(\,\cdot\,,\,\cdot\,)$ the real inner product) where $A$ is a complex symmetric matrix with positive-definite real part.  This may sound simple, but it is not: the real and imaginary parts of $A$ need not commute, making the problem extremely computationally difficult.  The lengthy calculations in Section \ref{section Proof of Main Theorem} in the special case $n=1$ attest to this; in fact, our approach is to first reduce to the $n=1$ case, in Section \ref{sect Segal}, as the general $n$-dimensional optimization does not admit a simple solution.
\end{remark}

\subsection{Identifying the Dual Space} \label{section duality 2} Since the projection $Q_\a$ is bounded on $L^p(\C^n,\l)$ for each $\a>0$, and has range equal to the space $\mathcal{S}_\a^p$, we can use the usual duality relations in $L^p$-spaces to translate to duality comparisons in $\mathcal{S}_\a^p$.  For ease of reading, denote by $\|Q_\a\|_{p\to p}$ the norm of $Q_\a\colon L^p(\l)\to \mathcal{S}_\a^p$.

\begin{lemma} \label{lemma duality in Sp} Let $\a>0$ and $1<p<\infty$. Let $p'$ denote the conjugate exponent to $p$.  In terms of the pairing $(F,G)_\l = \int_{\C^n} F\overline{G}\,d\l$, the spaces $\mathcal{S}_\a^p$ and $\mathcal{S}_\a^{p'}$ are dual, with
\begin{equation} \label{eq duality estimate 1}
\|Q_\a\|_{p\to p}^{-1} \|G\|_{\mathcal{S}_\a^{p'}} \le \|(\,\cdot\,,G)_\l\|_{(\mathcal{S}_\a^p)^\ast} \le \|G\|_{\mathcal{S}_\a^{p'}}
\end{equation}
for all $G\in\mathcal{S}_\a^{p'}$.
\end{lemma}

\begin{proof} As in the proof of Corollary \ref{cor P_a not bounded}, for $\Phi\in(\mathcal{S}_\a^p)^\ast$ denote by $\hat{\Phi}$ the linear functional $\hat{\Phi} = Q_\a^\ast\Phi$; i.e.\ $\hat{\Phi}(F) = \Phi(Q_\a F)$ for $F\in L^p(\l)$.  Since $Q_\a\colon L^p(\l)\to \mathcal{S}_\a^p$ is bounded, the linear functional $\hat\Phi$ is continuous on $L^p(\l)$; that is, $\hat\Phi\in(L^p(\l))^\ast$.  The  $L^p$ Riesz Representation Theorem therefore shows that there is a unique function $G\in L^{p'}(\l)$ such that $\hat\Phi = (\,\cdot\,,G)_\l$, and moreover that
\begin{equation} \label{eq standard Lp duality} \|(\,\cdot\,,G)_\l\|_{(L^p(\l))^\ast} = \|G\|_{L^{p'}(\l)}. \end{equation}
Since $Q_\a^2 = Q_\a$, it follows that $\hat\Phi(Q_\a F) = \hat\Phi(F)$ for all $F\in L^p(\l)$.  Thus, for any such $F$,
\begin{equation} \label{eq duality in Sp 1} (F,G)_\l = \hat\Phi(F) = \hat\Phi(Q_\a F) = (Q_\a F,G)_\l = (F,Q_\a G)_\l, \end{equation}
where the last equality holds since $Q_\a$ is self adjoint on $L^2(\l)$ (and so by a standard density argument the self-adjointness extends to the $L^p$--$L^{p'}$ pairing).  Since Equation \eqref{eq duality in Sp 1} holds for all $F\in L^p(\l)$, it follows that $G=Q_\a G$, and hence $G\in \mathcal{S}_\a^{p'}$.  Now, for $F\in \mathcal{S}_\a^p\subset L^p(\l)$, $\hat\Phi(F) = \Phi(Q_\a F) = \Phi(F)$; that is, $\hat\Phi|_{\mathcal{S}_\a^p} = \Phi$.  Thus, $\Phi(F) = (F,G)_\l$.

\medskip

To summarize, we have shown that the map $G\mapsto (\,\cdot\,,G)_{\l}$ is surjective from $\mathcal{S}_\a^{p'}$ onto $(\mathcal{S}_\a^p)^\ast$; it is also continuous due to Equation \eqref{eq standard Lp duality}.  We must now show that it is injective.  Since $G\in\mathcal{S}_\a$, by definition there is a holomorphic $g$ such that $G = \mathfrak{g}_\a g$.  Consider the monomial $f(z) = z^{\mx{j}}$ for $\mx{j}\in\N^n$.  Since $f$ has polynomial growth, $\mathfrak{g}_\a f\in\mathcal{S}_\a^p$ (for any $p$).  Hence, we can compute
\[ (\mathfrak{g}_\a f,\mathfrak{g}_\a g)_\l = \int z^{\mx{j}}\overline{g(z)}e^{-\a|z|^2}\,\l(dz) = \left(\frac{\a}{\pi}\right)^n\int z^\mx{j} \overline{g(z)}\,\g_\a(dz) = \left(\frac{\a}{\pi}\right)^n (f,g)_\a. \]
Due to the orthogonality relations of Equation \eqref{eq orth rel}, the inner product $(f,g)_\a$ is a scalar multiple of the $\mx{j}$th Taylor coefficient of $f$.  Since $g$ is holomorphic, these coefficients determine $g$ uniquely, and so too $G$.  It follows that the map $G\mapsto (\,\cdot\,,G)_\l$ is injective.  Thus, $S^p_\alpha$ and $S^{p'}_\a$ are dual with respect to $(\cdot,\cdot)_\l$.

As for Equation \eqref{eq duality estimate 1}, the first inequality follows since $\|G\|_{L^{p'}(\l)} = \|\hat\Phi\|_{(L^p(\l))^\ast}$, and
\[ \|\hat\Phi\|_{(L^p(\l))^\ast}
                   =  \sup_{F\in L^p(\l)} \frac{|\hat\Phi(F)|}{\|F\|_p}
                   =  \sup_{F\in L^p(\l)} \frac{|\Phi(Q_\alpha F)|}{\|F\|_p} 
                   \le  \|Q_\alpha\|_{p\to p} \sup_{F\in L^p(\l)} \frac{|\Phi(Q_\alpha F)|}{\|Q_\alpha F\|_p}, \]
where the inequality is just the statement that $\|Q_\a F\|_p \le \|Q_\a\|_p \|F\|_p$.  Note, as shown in the first paragraph, $\|\hat\Phi\|_{(L^p(\l))^\ast} = \|(\,\cdot\,,G)_\l\|_{(L^p(\l))^\ast} = \|G\|_{L^{p'}(\l)} = \|G\|_{\mathcal{S}_\a^{p'}}$ since $G\in\mathcal{S}_\a^{p'}$.  Since the range of $Q_\a$ on $L^p(\l)$ is all of $\mathcal{S}_\a^p$, we therefore have
\[  \|G\|_{\mathcal{S}_\a^{p'}} = \|\hat\Phi\|_{(L^p(\l))^\ast} \le         
                   \|Q_\alpha\|_{p\to p} \sup_{H\in\mathcal{S}^p_\a} \frac{|\Phi(H)|}{\|H\|_p} 
                 =  \|Q_\alpha\|_{p\to p}\|\Phi\|_{(\mathcal{S}^p_\a)^\ast} = \|Q_\a\|_p \|(\,\cdot\,,G)\|_{(\mathcal{S}^p_\a)^\ast}.
\]
The second estimate in Equation \eqref{eq duality estimate 1}, that $\|(\,\cdot\,,G)_\l\|_{(\mathcal{S}_\a^p)^\ast} \le \|G\|_{\mathcal{S}_\a^{p'}}$, is a straightforward consequence of H\"older's inequality.
\end{proof}

Due to the isometry $\mathfrak{g}_{\a,p}\colon\mathcal{S}_\a^p\to L^p_{hol}(\g_\a)$, Lemma \ref{lemma duality in Sp} also yields the proof of Theorem \ref{main corollary}.

\begin{proof}[Proof of Theorem \ref{main corollary}] Dividing through by the $L^p(\gamma_{\alpha p/2})\to L^p_{hol}(\gamma_{\alpha p/2})$ norm of $P_\alpha$ (following Theorem \ref{main theorem 1}), the desired Inequalities \eqref{eq main corollary 2} can be written as
\begin{equation} \label{eq duality estimate 2}
\|P_\a\|_{p\to p}^{-1}\|g\|_{L^{p'}_{hol}(\g_{\a p'/2})} \le \left(\textstyle{\frac12}p^{1/p} p'^{1/p'}\right)^n\|(\,\cdot\,,g)_\a\|_{L^p_{hol}(\g_{\a p/2})^\ast} \le \|g\|_{L^{p'}_{hol}(\g_{\a p'/2})}
\end{equation}
We now prove Inequalities \eqref{eq duality estimate 2}.  From Equation \eqref{eq norm Pa Qa}, $\|P_\a\|_{p\to p} = \|Q_\a\|_{p\to p}$.  Any $g\in L^{p'}_{hol}(\g_{\a p'/2})$ can be written uniquely as $g = \mathfrak{g}_{\a,p'}^{-1}G$ for some $G\in\mathcal{S}_\a^p$, and since $\mathfrak{g}_{\a,p'}$ is an isometry, Equation \eqref{eq duality estimate 1} then yields
\begin{equation} \label{eq duality estimate 1.5} \|P_\a\|_{p\to p}^{-1}\|g\|_{L^{p'}_{hol}(\g_{\a p'/2})} = \|Q_\a\|_{p\to p}^{-1} \|G\|_{S_\a^{p'}} \le \|(\,\cdot\,,G)_\l\|_{(\mathcal{S}_\a^p)^\ast} \le \|G\|_{\mathcal{S}_\a^{p'}} = \|g\|_{L^{p'}_{hol}(\g_{\a p'/2})}.  \end{equation}
We are left to re-express $\|(\,\cdot\,,G)_\l\|_{(\mathcal{S}_\a^p)^\ast}$ in terms of $g$. Since $\mathfrak{g}_{\a,p}\colon\mathcal{S}_\a^p\to L^p_{hol}(\g_{\a p/2})$ is an isometric isomorphism, for any $F\in \mathcal{S}_\a^p$ there is a unique $f\in L^p_{hol}(\g_{\a p/2})$ such that $F = \mathfrak{g}_{\a,p}f$.  Taking $G=\mathfrak{g}_{\a,p'}g$ as above, we can write
\begin{equation} \label{eq dual norm 1} \|(\,\cdot\,,G)_\l\|_{(\mathcal{S}_\a^p)^\ast} = \sup_{F} \frac{|(F,G)_\l|}{\|F\|_{\mathcal{S}_\a^p}} 
= \sup_f \frac{|(\mathfrak{g}_{\a,p}f,\mathfrak{g}_{\a,p'}g)_\l|}{\|f\|_{L^p_{hol}(\g_{\a p/2})}}. \end{equation}
From Equations \eqref{eq gg_a} and \eqref{eq g_ap} defining the isometry $\mathfrak{g}_{\a,p}$, we have
\[ \begin{aligned} (\mathfrak{g}_{\a,p}f,\mathfrak{g}_{\a,p'}g)_\l &= \int_{\C^n} \left(\frac{p\a}{2\pi}\right)^{n/p} e^{-\frac{\a}{2}|z|^2}f(z) \left(\frac{p'\a}{2\pi}\right)^{n/p'}e^{-\frac{\a}{2}|z|^2} \overline{g(z)}\,\l(dz) \\
&= \left(\frac{ p^{1/p} p'^{1/p'} \alpha}{2\pi}\right)^n\int_{\C^n} f(z)\overline{g(z)} e^{-\a|z|^2}\,\l(dz)
= \left(\textstyle{\frac12}p^{1/p}p'^{1/p'}\right)^n (f,g)_\a.
\end{aligned} \]
Hence Equation \eqref{eq dual norm 1} becomes
\begin{equation} \label{eq dual norm 2}
\|(\,\cdot\,,G)_\l\|_{(\mathcal{S}_\a^p)^\ast} = \left(\textstyle{\frac12}p^{1/p}p'^{1/p'}\right)^n\sup_f\frac{|(f,g)_\a|}{\|f\|_{L^p_{hol}(\g_{\a p/2})}} = \left(\textstyle{\frac12}p^{1/p}p'^{1/p'}\right)^n \|(\,\cdot\,,g)_\a\|_{(L^p_{hol}(\g_{\a p/2}))^\ast}.
\end{equation}
Equations \eqref{eq duality estimate 1.5} and \eqref{eq dual norm 2} combine to prove the estimates in Equation \eqref{eq duality estimate 2}.

%

\end{proof}


%
%
\section{The Norm of $P_\alpha$} \label{section Proof of Main Theorem}

In this section, we prove Theorem \ref{main theorem 1}.  Since the sharp constant in that theorem is an $n$th power, we begin by showing that it is sufficient to prove the theorem in the case $n=1$; this is a version of an idea due to Segal.  The kernel $\mathcal{Q}_\alpha$ is a Gaussian function, and so to find the norm of $P_\alpha$ (which is the same as that of $Q_\alpha$), we use the deep results of \cite{Lieb} to reduce to a calculation on putative Gaussian maximizers.  The resulting optimization problem is difficult; the majority of this section is devoted to its resolution in the $n=1$ case.  (Note that $n$ is the complex dimension; the maximizer thus corresponds to a $2\times 2$ complex symmetric matrix, so the calculation involves a complicated function of $6$ real variables.)

\subsection{Segal's Lemma} \label{sect Segal} Proposition \ref{prop Segal} below is a simple variant of what is colloquially known as {\em Segal's Lemma}, based on a version appearing as Lemma 1.4 in \cite{Segal} (for the case of positive kernels $G_i$).  The proposition actually holds for kernels mapping $L^p\to L^q$ for $1<p\le q<\infty$: the following proof need only be modified by replacing the application of Tonelli's theorem in Equation \eqref{eq Segal 3} with Minkowski's inequality for integrals.  The proof we present is essentially contained in the proof of \cite[Theorem 3.3]{Lieb}.

\begin{proposition} \label{prop Segal} Let $n,m\ge 1$ be integers, let $1\le p<\infty$, and let $p'$ denote the conjugate exponent to $p$.   Let $G_1\colon\R^{n}\times\R^{n}\to\C$ and $G_2\colon\R^{m}\times\R^{m}\to\C$ be complex functions, such that $G_1(x_1,\cdot) \in L^{p'}(\R^n)$ for almost every $x_1\in\R^n$ and $G_2(x_2,\cdot)\in L^{p'}(\R^m)$ for almost every $x_2\in\R^m$.  Define
\[ T_1f(x_1) = \int_{\R^n} G_1(x_1,y_1)f(y_1)\,dy_1, \qquad T_2f(x_2) = \int_{\R^m} G_2(x_2,y_2)f(y_2)\,dy_2. \]
Let $G = G_1\tensor G_2$: $G((x_1,x_2),(y_1,y_2)) = G_1(x_1,y_1)G_2(x_2,y_2)$; and let $T = T_1\tensor T_2$:
\[ TF(x_1,x_2) = \int_{\R^n}\int_{\R^m} G((x_1,x_2),(y_1,y_2))F(y_1,y_2)\,dy_1dy_2. \]
If $T_1$ is bounded on $L^p(\R^n)$ and $T_2$ is bounded on $L^p(\R^m)$, then $T$ is bounded on $L^p(\R^n\times\R^m)$, and
\[ \|T\|_{p\to p} = \|T_1\|_{p\to p}\|T_2\|_{p\to p}. \]
\end{proposition}

\begin{proof} Let $F\in L^p(\R^n\times\R^m)$.  For $1<p<\infty$, by Tonelli's theorem, for almost every $(x_1,x_2)\in\R^n\times\R^m$ we have
\begin{align*} \|G((x_1,x_2),(\cdot,\cdot))\|_{p'}^{p'} &= \int_{\R^n\times\R^m} |G_1(x_1,y_1)G_2(x_2,y_2)|^{p'}\,dy_1dy_2 \\
&= \int_{\R^n} \left(\int_{\R^m} |G_1(x_1,y_1)|^{p'} |G_2(x_2,y_2)|^{p'}\, dy_2\right)dy_1 \\
&= \left(\int_{\R^n} |G_1(x_1,y_1)|^{p'}\,dy_1\right)\left(\int_{\R^m} |G_2(x_2,y_2)|^{p'}\,dy_2\right) < \infty
\end{align*}
by assumption.  Similarly, in the case $p=1$ it is easy to check that $G\in L^\infty(\R^n\times\R^m)$.  Thus, using H\"older's inequality,
\[ \int_{\R^n\times\R^m} |G_1(x_1,y_1)G_2(x_2,y_2)F(y_1,y_2)|\,dy_1dy_2 \le \|G((x_1,x_2),(\cdot,\cdot))\|_{p'} \|F\|_p <\infty \]
for almost every $(x_1,x_2)\in\R^n\times\R^m$.  Hence, calculating the $L^p$ norm of $TF$, we can apply Fubini's theorem:
\begin{align} \nonumber  \|TF\|_p^p &=  \int_{\R^n\times\R^m} \left|\int_{\R^n\times\R^m} G_1(x_1,y_1)G_2(x_2,y_2)F(y_1,y_2)\,dy_1dy_2\right|^p dx_1dx_2 \\
\nonumber &= \int_{\R^n} \left( \int_{\R^m} \left| \int_{\R^m}G_2(x_2,y_2)\left(\int_{\R^n} G_1(x_1,y_1)F(y_1,y_2)\,dy_1\right) dy_2\right|^p dx_2 \right)dx_1 \\
\label{eq Segal 1} &= \int_{\R^n} \left( \int_{\R^m} \left| \int_{\R^m}G_2(x_2,y_2)K(x_1,y_2)\, dy_2\right|^p dx_2 \right)dx_1 \end{align}
where
\begin{equation} \label{eq Segal 2} K(x_1,y_2) = \int_{\R^n} G_1(x_1,y_1)F(y_1,y_2)\,dy_1 = (T_1F(\cdot,y_2))(x_1). \end{equation}
We do not know, a priori, whether the function $K(x_1,\cdot)$ is in $L^p(\R^m)$; but it is nevertheless true that
\[ \int_{\R^m}\left|\int_{\R^m} G_2(x_2,y_2)K(x_1,y_2)\,dy_2\right|^p dx_2 = \| T_2K(x_1,\cdot)\|_p^p \le (\|T_2\|_{p\to p})^p \|K(x_1,\cdot)\|_p^p \]
for each $x_1\in\R^n$ (where the right-hand-side is $+\infty$ in the case $K(x_1,\cdot)\notin L^p$).  Combining with Equation \eqref{eq Segal 1} then yields
\begin{align} \nonumber \|TF\|_p^p = \int_{\R^n}  \| T_2K(x_1,\cdot)\|_p^p\, dx_1 &\le  (\|T_2\|_{p\to p})^p \int_{\R^n} \|K(x_1,\cdot)\|_p^p\, dx_1 \\
\label{eq Segal 3} &= (\|T_2\|_{p\to p})^p \int_{\R^n} \left( \int_{\R^m} |K(x_1,y_2)|^p\,dy_2\right) dx_1.
\end{align}
We now apply Tonelli's theorem to the (non-negative) integrand in Equation \eqref{eq Segal 3}, to reverse the order of integration:
\begin{equation} \label{eq Segal 4} \int_{\R^n} \left( \int_{\R^m} |K(x_1,y_2)|^p\,dy_2\right) dx_1 = \int_{\R^m} \left( \int_{\R^n} |K(x_1,y_2)|^p\,dx_1\right) dy_2. \end{equation}
For almost every $y_2\in\R^m$, the function $F(\cdot,y_2)$ is in $L^p(\R^n)$.  Referring to Equation \eqref{eq Segal 2}, it follows that for such $y_2$ we have
\[ \int_{\R^n} |K(x_1,y_2)|^p\,dx_1 = \|T_1F(\cdot,y_2)\|_p^p \le (\|T_1\|_{p\to p})^p \|F(\cdot,y_2)\|_p^p. \]
Thus, Equation \eqref{eq Segal 4} gives
\begin{align} \nonumber \int_{\R^m} \left( \int_{\R^n} |K(x_1,y_2)|^p\,dx_1\right) dy_2 &\le (\|T_1\|_{p\to p})^p \int_{\R^m} \|F(\cdot,y_2)\|_p^p\,dy_2 \\
\nonumber &= (\|T_1\|_{p\to p})^p \int_{\R^m} \left(\int_{\R^n} |F(y_1,y_2)|^p\,dy_1\right)dy_2 \\
\label{eq Segal 5} &= (\|T_1\|_{p\to p})^p \|F\|_p^p.
\end{align}
Equation \eqref{eq Segal 5} shows that $K\in L^p(\R^n\times\R^m)$, and so in particular the bound in Equation \eqref{eq Segal 3} is non-trivial.  Combining Equations \eqref{eq Segal 3}, \eqref{eq Segal 4}, and \eqref{eq Segal 5} and taking $p$th roots shows that
\[ \|TF\|_p \le \|T_2\|_{p\to p}\|T_1\|_{p\to p} \|F\|_p \]
proving the upper bound of the proposition.  The lower bound is actually the easier direction: let $f_k\in L^p(\R^n)$ and $g_k\in L^p(\R^m)$ be $L^p$-normalized functions saturating the $L^p$-norms of the operators $T_1$ and $T_2$; that is, $\|f_k\|_p = \|g_k\|_p = 1$ and
\[ \lim_{k\to \infty}\|T_1f_k\|_p = \|T_1\|_{p\to p}, \quad \lim_{k\to\infty}\|T_2g_k\|_p = \|T_2\|_{p\to p}. \]
Let $F_k = f_k\tensor g_k$: $F_k(y_1,y_2) = f_k(y_1)g_k(y_2)$.  Then Tonelli's theorem quickly shows that $\|F_k\|_p = \|f_k\|_p\|g_k\|_p = 1$, and Fubini's theorem (as above) shows that
\begin{align*} TF_k(x_1,x_2) &= \int_{\R^n\times\R^m} G_1(x_1,y_1)G_2(x_2,y_2)f_k(y_1)g_k(y_2)\,dy_1dy_2 \\
&= \left(\int_{\R^n} G_1(x_1,y_1)f_k(y_1)\,dy_1\right)\left(\int_{\R^m} G_2(x_2,y_2)g_k(y_2)\,dy_2\right) = T_1f_k(x_1)\cdot T_2g_k(x_2).
\end{align*}
Then Tonelli's theorem (as above) shows that
\[ \|TF_k\|_p = \|T_1f_k\|_p\|T_2g_k\|_p \to \|T_1\|_{p\to p} \|T_2\|_{p\to p}\;\; \text{  as  }\;\; k\to\infty. \]
This shows that $\|T\|_{p\to p} \ge \|T_1\|_{p\to p}\|T_2\|_{p\to p}$, completing the proof.
\end{proof}

\begin{remark} The kernel $\mathcal{Q}_\alpha$ is, in fact, a tensor power; induction on Proposition \ref{prop Segal} will therefore reduce the calculation of $\|Q_\alpha\|_{p\to p}$ to the ($n$th power of the) $n=1$ case.  First we need to verify that $\mathcal{Q}_\alpha$ satisfies the $L^{p'}$-bound conditions of Proposition \ref{prop Segal}.  This will follow easily from the Gaussian character of the kernel, and is the content of Corollary \ref{cor Segal}.  Indeed, this Gaussian character gives us much more, as the next section attests to.
\end{remark}

\begin{corollary} \label{cor Segal} Let $Q_\alpha^1$ denote the operator $Q_\alpha$ of Equation \eqref{eq Q_a} in the case $n=1$.  Let $1<p<\infty$, and $\alpha>0$.  Then
\[ \|Q_\alpha\|_{p\to p} = \left(\|Q_\alpha^1\|_{p\to p}\right)^n. \]
\end{corollary}

\begin{proof} Note, from Equation \eqref{eq Q_a}, that
\[ Q_\alpha F(x) = \int_{\R^{2n}} \mathcal{Q}_\a(z,w)F(w)\,dw \]
where
\[ \mathcal{Q}_\a(z,w) = \left(\frac{\alpha}{\pi}\right)^n\exp\left\{-\frac{\alpha}{2}\left(|z|^2+|w|^2-2\langle z,w\rangle\right)\right\}. \]
Of course the quadratic form is a sum over independent variables,
\[ |z|^2+|w|^2-2\langle z,w\rangle = \sum_{j=1}^n \left(|z_j|^2 + |w_j|^2 - 2z_j\overline{w_j}\right) \]
and so we have
\[ \mathcal{Q}_\a(z,w) = \prod_{j=1}^n \mathcal{Q}_\a^1(z_j,w_j) \]
where
\[ \mathcal{Q}_\a^1(z_1,w_1) = \frac{\alpha}{\pi} \exp\left\{-\frac{\alpha}{2}\left(|z_1|^2 + |w_1|^2 -2z_1\overline{w_1}\right)\right\}. \]
Notice also that $\mathcal{Q}_\a^1$ is the kernel of $Q_\alpha^1$; so we have
\[ Q_\alpha = \bigotimes_{j=1}^n Q_\alpha^1. \]
Furthermore, the kernel $\mathcal{Q}_\alpha$ (in any dimension) satisfies
\[ |\mathcal{Q}_\a(z,w)| = \left(\frac{\alpha}{\pi}\right)^n \exp\left\{-\frac{\alpha}{2}(|z|^2+|w|^2-2\Re\langle z,w\rangle)\right\} = \left(\frac{\alpha}{\pi}\right)^n \exp\left\{-\frac{\alpha}{2}|z-w|^2\right\} \]
(as computed once before in Equation \eqref{eq |Qa|}).  Of course this means $\mathcal{Q}_\a(z,\cdot) \in L^\infty$ (with norm $(\alpha/\pi)^n$) and also in $L^{p'}$ for any $p>1$; using Equation \eqref{gaussianintegral},
\[ \int |\mathcal{Q}_\a(z,w)|^{p'}\,dw = \left(\frac{\alpha}{\pi}\right)^{np'} \int_{\R^{2n}} e^{-\frac{\alpha p'}{2}|z-w|^2}dw = 
\left(\frac{\alpha}{\pi}\right)^{np'} \left(\frac{2\pi}{\alpha p'}\right)^n < \infty \]
for all $z$.  Hence, the corollary follows by induction on Proposition \ref{prop Segal}.
\end{proof}

\subsection{Gaussian Kernels} \label{section Gaussian kernels}

As previously proved, the projection $Q_\alpha$ is a bounded map $L^p(\C^n,\l)\to \mathcal{S}^p_\alpha$, with 
\[ \|Q_\alpha\|_{p\to p} = \|Q_\alpha\colon L^p(\C^n,\l)\to \mathcal{S}^p_\alpha \|=\|P_\alpha\colon L^p(\C^n,\g_{\a p/2}) \to L^p_{hol}(\C^n,\g_{\a p/2})\|\le 2^n. \]
(cf. Theorem \ref{main theorem 1} and Proposition \ref{prop commutative diagram}). 
We will investigate the bound $\|Q_\alpha\|_{p\to p}$ (and thus $\|P_\alpha\|_{p\to p}$) using a main result of \cite{Lieb}.  Before we state the particular theorem from \cite{Lieb} that we will use, we will first establish some notation.  For a fixed integer $k\geq 1$, define the set of matrices $\mathcal{A}^k$ as 
\begin{equation}\label{defmathcalA}
\mathcal{A}^k=\{A\in \mathbb{C}^{k\times k}:\mbox{$A$ is symmetric and $\Re(A)$ is positive definite}\}.
\end{equation}
In turn we define the set of \emph{(centered) Gaussian functions} $\mathcal{G}^k$ as 
\begin{equation}\label{defmathcalG}
\mathcal{G}^k=\{g(x)=e^{-(x,Ax)}: A\in\mathcal{A}^k\}.
\end{equation}
In the definition above and in what follows, the inner product $\left(\cdot,\cdot\right)$ denotes the standard inner product on $\R^k$ extended to $\C^k$ such that $\left(\cdot,\cdot\right)$ is bilinear (and not sesquilinear). We can now state the theorem will we use from \cite{Lieb}:
\begin{theorem}[Lieb, 1990]\label{LiebResult}
Let $1<p<\infty$.  Suppose $T:L^p(\R^{k},\l)\to L^p(\R^{k},\l)$ is a bounded integral operator with a Gaussian kernel $G(x,y)$.  Specifically, for $f\in L^p(\R^{k},\l)\cap L^1(\R^{k},\l)$, we can write $T(f)(x)$ as
$$T(f)(x)=\int_{\mathbb{R}^k}f(y)G(x,y)\,dy,$$
where $G(x,y)$ has the form
$$G(x,y)=\exp\{-(x,D_{11}x)-(y,D_{22}y)-2(x,D_{12}y)\},$$
where $D_{11}$ and $D_{22}$ are real symmetric matrices.  
If the real part of the block matrix $\left[\begin{array}{cc}D_{11} & D_{12} \\D_{12}^T & D_{22} \end{array}\right]$ is positive semidefinite, then the norm $\|T\|$ can be computed as
$$\|T\|=\sup_{g\in\mathcal{G}^k}\frac{\|Tg\|_p}{\|g\|_p}.$$
\end{theorem}
Theorem \ref{LiebResult} is a less general version of Theorem 4.1 in \cite{Lieb}.  To apply Theorem \ref{LiebResult} to $Q_\alpha$, we will need to view the space $L^p(\C^n,\l)$ as $L^p(\R^{2n},\l)$.  Recall that
\begin{equation} \label{eq Qa general n}
Q_\alpha F(z) =  \left(\frac{\alpha}{\pi}\right)^n\int_{\C^n} F(w)e^{-\alpha|z|^2/2-\alpha|w|^2/2+\alpha\langle z,w\rangle}dw,
\end{equation}
where $\langle\cdot,\cdot\rangle$ denotes the sesquilinear inner product which is linear in its first argument.  Associate $z=x_1+ix_2$ with the real vector $x=[x_1, x_2]$ and $w=y_1+iy_2$ with the real vector $y=[y_1, y_2]$.  
$Q_\alpha$ becomes
\begin{equation}\label{def1}
Q_\alpha F(x) = \left(\frac{\alpha}{\pi}\right)^n \int_{\R^{2n}}  F(y)e^{-(x,D_{11}x)-(y,D_{22}y)-2(x,D_{12}y)}dy,
\end{equation}
where $D_{11}=D_{22}=(\alpha/2) I_{2n}$ and 
\begin{eqnarray*}
D_{12}=-\alpha/2\left[\begin{array}{cc}I_n & -iI_n \\iI_n & I_n \end{array}\right].
\end{eqnarray*} 
One can check that the real part of the $4n$-by-$4n$ matrix $\left[\begin{array}{cc}D_{11} & D_{12} \\D_{12}^T & D_{22} \end{array}\right]$ has exactly two eigenvalues $0$ and $\alpha$, each of multiplicity $2n$.  Since $\alpha$ is assumed to be positive, the block matrix is positive semidefinite. 
Thus, we can apply Theorem \ref{LiebResult} and we have
\begin{lemma}\label{LiebandQalpha}
For any $1<p<\infty$, the operator norm $\|Q_\alpha\|_{p\to p}$ of $Q_\alpha:L^p(\mathbb{R}^{2n},\l)\to \mathcal{S}^p_\alpha\subset L^p(\mathbb{R}^{2n},\l)$ is
\begin{equation}\label{normratio}
\|Q_\alpha\|_{p\to p}=\sup_{g\in\mathcal{G}^{2n}}\frac{\|Q_\alpha g\|_p}{\|g\|_p}.
\end{equation}
\end{lemma}
%
%
%

\subsection{A Formula for $\frac{\|Q_\alpha g\|_p}{\|g\|_p}$} \label{section formula for ratio}

\begin{lemma}\label{lemmainnormquotient1}
Let $n$ be any positive integer, $0<p<\infty$, and $g\in\mathcal{G}^{2n}$ so that $g(x)=e^{-(x,Ax)}$ for some $A\in\mathcal{A}^{2n}$.  Let $A'=\frac{2}{\alpha}A$.  Then 
\begin{equation}
\frac{\|Q_\alpha g\|_p^p}{\|g\|_p^p}=2^{np}\sqrt{\frac{\det(\Re(A'))}{|\det(A'+I_{2n})|^p\det(I_{2n}+\Omega((A'+I_{2n})^{-1}))}},\label{mainnormquotient}
\end{equation}
where for any matrix $M\in\mathbb{C}^{2n\times 2n}$
\begin{equation}
\Omega(M):=J^T\Re(M)J-\Re(M)-\Im(M)J-J^T\Im(M)\label{defOmega}.
\end{equation}
In Equation \eqref{defOmega}, $J$ is the $2n\times 2n$ symplectic matrix $J:=\left[\begin{array}{cc}0 & -I_n \\I_n & 0 \end{array}\right]$.
\end{lemma}
Using $A'$ instead of $A$ above allows us to write the quotient $\frac{\|Q_\alpha g\|_p^p}{\|g\|_p^p}$ independently of $\alpha$.
\begin{proof}
Note that $g(x)=e^{-\left(x,\frac{\alpha}{2}A'x\right)}$.   Using Equation \eqref{gaussianintegral} above we have
\begin{equation}
\|g\|_p^p
=\left(\frac{2\pi}{p\alpha}\right)^n\frac{1}{\sqrt{\det(\Re(A'))}}.\label{normg}
\end{equation}
To calculate $\|Q_\alpha g\|_p^p$, note that $D_{12}=-\frac{\alpha}{2}(I_{2n}+iJ)$ and $J^T = -J$; using this together with Equation \eqref{gaussianintegral}, one can calculate
\begin{eqnarray}
\|Q_\alpha g\|_p^p
&=& 
\left(\frac{2^{np}}{\sqrt{|\det(A'+I_{2n})|}^p}\right)\cdot\nonumber\\
& & \int_{\R^{2n}}e^{-\frac{p\alpha}{2}(x,x)}\left|e^{(x,\frac{\alpha}{2}(I_{2n}+iJ)(A'+I_{2n})^{-1}(I_{2n}+iJ^T)x)}\right|^p dx.\label{Qalphanorm1}
\end{eqnarray}
One can verify that
\begin{equation}
\Re((I_{2n}+iJ)(A'+I_{2n})^{-1}(I_{2n}-iJ))
=-\Omega((A'+I_{2n})^{-1}).\label{Omega1}
\end{equation}
Thus, plugging in Equation \eqref{Omega1} into Equation \eqref{Qalphanorm1} and again applying Equation \eqref{gaussianintegral} yields
\begin{eqnarray}
\|Q_\alpha g\|_p^p
&=&\left(\frac{2^{np}}{\sqrt{|\det(A'+I_{2n})|}^p}\right)\int_{\R^{2n}}e^{-\frac{p\alpha}{2}(x,x)}e^{(x,-\frac{p\alpha}{2}\Omega((A'+I_{2n})^{-1})x)} dx\nonumber\\
&=&\left(\frac{\left(\frac{2\pi}{p\alpha}\right)^n2^{np}}{\sqrt{|\det(A'+I_{2n})|^p\det(I_{2n}+\Omega((A'+I_{2n})^{-1}))}}\right)\label{Qalphanorm2}.
\end{eqnarray}
Dividing \eqref{Qalphanorm2} by \eqref{normg} gives the lemma. 
\end{proof}
Now we have a new characterization of the norm $\|Q_\alpha\|_{p\to p}$.
\begin{lemma}\label{lemmainnormquotient2}  Let $1<p<\infty$.  Then
$$(\|Q_\alpha\|_{p\to p})^p=2^{np}\sup_{A\in\mathcal{A}^{2n}}\sqrt{\frac{\det(\Re(A))}{|\det(A+I_{2n})|^p\det(I_{2n}+\Omega((I_{2n}+A)^{-1}))}}.$$
\end{lemma}
\begin{proof}
Note that the mapping $A\to\frac{2}{\alpha}A$ is a bijection of $\mathcal{A}^{2n}$ to itself.  Thus, combining Lemmas \ref{LiebandQalpha} and \ref{lemmainnormquotient1} gives the result.
\end{proof}

\subsection{The Optimization Problem} \label{section Optimization}  Lemma \ref{lemmainnormquotient2} reduces the determination of the sharp norm of $Q_\alpha$ to an optimization problem over the space $\mathcal{A}^{2n}$ of $2n\times 2n$ complex symmetric matrices with positive definite real part.  While more tractable than the general optimization over $L^p$, the domain of this function is an open subset of a $2n(2n+1)$ (real) dimensional space, and the function is quite complicated.  The task of identifying all critical points of this function in general is quite difficult.  Instead, we use Proposition \ref{prop Segal} to reduce to the case $n=1$, where we devote the remainder of this paper to an analysis of the optimization.  In particular, Corollary \ref{cor Segal} and Lemma \ref{lemmainnormquotient2} show that for general $n$
\begin{equation}
(\|Q_\alpha\|_{p\to p})^p=\left(2^{p}\sup_{A\in\mathcal{A}^2}\sqrt{\frac{\det(\Re(A))}{|\det(A+I_{2})|^p\det(I_{2}+\Omega((I_{2}+A)^{-1}))}}\right)^{n} \label{normeq},
\end{equation} 
where 
\begin{equation*}
\Omega(M)=J^T\Re(M)J-\Re(M)-J^T\Im(M)-\Im(M)J,
\end{equation*}
$I_2$ is the $2\times 2$ identity matrix, and $J$ is the counter-clockwise rotation by $90^\circ$
\[ J=\left[\begin{array}{cc} 0 & -1 \\ 1 & 0  \end{array}\right]. \]
Accordingly, we define the function $h_p\colon\mathcal{A}^2\to\R$ by
\begin{equation} \label{eq hp def}
h_p(A):=\frac{\det(\Re(A))}{|\det(A+I_{2})|^p\det(I_{2}+\Omega((I_{2}+A)^{-1}))}.
\end{equation}
From Equation \eqref{normeq} and the lower bound of Proposition \ref{prop elementary bounds}, and since $\|P_\alpha\|_{p\to p} = \|Q_\alpha\|_{p\to p}$ by Equation \eqref{eq norm Pa Qa},  it therefore follows that to prove Theorem \ref{main theorem 1}, it suffices to show that
\begin{equation} \label{eq hp bound 1} h_p(A) \le \left(\frac{1}{p^{1/p} p'^{1/p'}}\right)^{2p}, \quad \forall \; A\in \mathcal{A}^2. \end{equation}
What's more, we already showed (in the proof of Proposition \ref{prop elementary bounds}) that Inequality \eqref{eq hp bound 1} holds for matrices of the form $A = \beta I_2$; indeed, when $p>1$, $h_p(\beta I_2)$ is maximized at its unique critical point $\beta = \frac{1}{p-1}$, where $h_p$ indeed takes the desired value.  This suggests the outline of the remainder of this section: we will show that, on the $6$-dimensional open set $\mathcal{A}^2$, the function $h_p$ has the unique critical point $A_p= \frac{1}{p-1}I_2$, and achieves its maximum there.  Note that since $Q_\a$ is an orthogonal projection on $L^2$, the norm is already known to be $1$ in that case, so in the sequel we consider only $p\ne2$ in $(1,\infty)$.

%
%
%
%
\subsection{The Critical Point of $h_p$} \label{section critical point}
The first step in looking for critical points is to write $h_p$ in terms of coordinates. So for $A\in\mathcal{A}^{2}$, we write
\begin{equation}\label{Acoordinates}
A=\left[\begin{array}{cc}a+ie&b+if\\b+if& d+ig\end{array}\right]\to (a,d,b,e,g,f) ,
\end{equation} 
and with this coordinate mapping consider $h_p$ as a function on an open subset of $(0,\infty)^2\times(-\infty,\infty)^4$.  The next lemma and the technical lemma that follows are long but straightforward calculations in coordinates.  We omit the proofs, but the results can be checked by hand or by a computer algebra system. 
\begin{lemma}\label{fpincoordinates}
The function $h_p:\mathcal{A}^2\to(0,\infty)$ can be written in coordinates as
\begin{equation*}
h_p(a,d,b,e,g,f)=\frac{ad-b^2}{(\Psi^2+\Phi^2)^{\frac{p-2}{2}}\tau},
\end{equation*}
where
\begin{eqnarray}
\tilde{a}:=a+1,\ & \tilde{d}=:d+1 \label{adtilde}\\
\Psi:=\tilde{a}\tilde{d}-b^2-eg+f^2, & \Phi:=\tilde{d}e+\tilde{a}g-2bf\label{PsiPhidef} \\
\psi:=a-d+2f, & \phi:=e-g-2b \label{psiphidef}\\
\tau:=\Psi^2+\Phi^2-\psi^2-\phi^2\label{taudef1}.
\end{eqnarray}
\end{lemma}
It will be useful to have the expression $\tau$ in Lemma \ref{fpincoordinates} written in the following forms form in the case that $b=0$; again, the elementary calculations are omitted.
\begin{lemma}
Consider the expression $\tau=\tau(a,d,b,e,g,f)$, defined as $\tau=\Psi^2+\Phi^2-\psi^2-\phi^2$.  When $b=0$, $\tau(a,d,0,e,f,g)$ can be written as
\begin{eqnarray}
\lefteqn{\tau(a,d,0,e,g,f)}\nonumber\\
&=&(ad-eg+f^2-1)^2+8ad+2ad(a+d)+2a(f-1)^2+\nonumber\\
& & 2d(f+1)^2+(ag+de)^2+2(de^2+ag^2)\label{tauform1}\\
&=& (eg-f^2+1)^2+a^2d^2+2ad(a+d)+6ad+2a(f-1)^2+\nonumber\\
& & 2adf^2+2d(f+1)^2+e^2(d^2+2d)+g^2(a^2+2a)\label{tauform2}.
\end{eqnarray}
\end{lemma}
We are now ready to compute partial derivatives of $h_p$ to look for a critical point.  First, we will consider critical points of a certain type.  More specifically, 
define a set of matrices $\mathcal{A}'^2\subset \mathcal{A}^2$ by
\begin{equation}\label{defA'}
\mathcal{A}'^2=\{A\in\mathcal{A}^2:\mbox{$\Re(A)$ is diagonal}\}.
\end{equation}
The fact that elements of $\mathcal{A}'^2$ have diagonal real parts will be quite useful in proving the following proposition, and it will turn out that proving statements on the set $\mathcal{A}'^2$ will lead to general results on $\mathcal{A}^2$.
As we mentioned in the introduction of this section, we will impose the condition that $p\neq 2$.  
\begin{proposition}\label{lemcriticalpointb=0}
Let $1<p<\infty$ and $p\neq 2$. Suppose $A\in\mathcal{A}'^{2}$ and a critical point of $h_p$.  Then $A=\frac{1}{p-1}I_{2}$.
\end{proposition}
\begin{proof}
We first derive a formula for $\frac{\partial h_p}{\partial x}$ where $x$ is any variable.  For an expression $\omega$, we let $\omega_x$ denote the partial derivative of $\omega$ with respect to $x$.  One can easily calculate
\begin{eqnarray*}
\frac{\partial h_p}{\partial x}
&=& \frac{(ad-b^2)[-(\Psi\Psi_x+\Phi\Phi_x)(p\tau+2(\psi^2+\phi^2))]}{(\Psi^2+\Phi^2)^{\frac{p}{2}}\tau^2}+\\
& & \frac{2(ad-b^2)(\psi\psi_x+\phi\phi_x)(\Psi^2+\Phi^2)+(\Psi^2+\Phi^2)(ad-b^2)_x\tau}{(\Psi^2+\Phi^2)^{\frac{p}{2}}\tau^2}.
\end{eqnarray*}
Let $A=\left[\begin{array}{cc}a+ie & if\\if & d+ig\end{array}\right]\in\mathcal{A}'^2$ be a critical point of $h_p$.  Then all six partial derivatives are 0, which imply the following six equations
\begin{eqnarray*}
0&=&ad\left[-\alpha(p\tau+2(\psi^2+\phi^2))+2\psi\right]+d\tau,\\
0&=&ad\left[-\delta(p\tau+2(\psi^2+\phi^2))-2\psi\right]+a\tau,\\
0&=&-\beta(p\tau+2(\psi^2+\phi^2))-2\phi,\, 0={\sigma(p\tau+2(\psi^2+\phi^2))+2\psi},\\
0&=&{\epsilon(p\tau+2(\psi^2+\phi^2))+2\phi},\, 0={\gamma(p\tau+2(\psi^2+\phi^2))-2\phi},
\end{eqnarray*}
where we define $\alpha,\delta,\beta,\sigma,\epsilon,\gamma$ as
$$(A+I_{2})^{-1}=\left[\begin{array}{cc}\alpha+i\epsilon&\beta+i\sigma\\ \beta+i\sigma & \delta+i\gamma\end{array}\right].$$
Since $p\tau+2(\psi^2+\phi^2)>0$, we can solve each equation for the corresponding value of $(A+I_{2})^{-1}$.  Define $C_p$ as 
$$C_p:=p\tau+2(\psi^2+\phi^2).$$  
Then the six equations above become
\begin{eqnarray}
\alpha=\frac{\tau}{aC_p}+\frac{2\psi}{C_p},& \displaystyle\delta = \frac{\tau}{dC_p}-\frac{2\psi}{C_p},& \beta = \frac{-2\phi}{C_p},\nonumber\\
\sigma = \frac{-2\psi}{C_p},& \displaystyle\epsilon = \frac{-2\phi}{C_p},& \gamma = \frac{2\phi}{C_p}.\label{coefficentequations}
\end{eqnarray}
Note that
$$\beta=\epsilon=-\gamma.$$
Thus we can write $(A+I_{2})^{-1}=\left[\begin{array}{cc}\alpha+i\beta&\beta+i\sigma\\ \beta+i\sigma & \delta-i\beta\end{array}\right]$
and thus $I_2=(A+I_2)(A+I_{2})^{-1}$ gives us the following eight equations:
\begin{eqnarray}
1=\ta\alpha-e\beta-f\sigma,& 1=\td\delta-f\sigma+g\beta\label{inveq1,7}\\
0= e\alpha+\ta\beta+f\beta,& 0=\ta\beta-e\sigma+f\beta\label{inveq2,3}\\
0=\td\beta-f\beta-g\sigma,& 0=-\td\beta+f\beta+g\delta\label{inveq5,8}\\
0=\ta\sigma+e\beta+f\delta,& 0=\td\sigma+f\alpha+g\beta\label{inveq4,6}
\end{eqnarray}
Subtracting the equations in \eqref{inveq2,3} yields $e(\alpha+\sigma)=0$.
So either $e=0$ or $\alpha+\sigma=0$.  If $\alpha+\sigma=0$, then
\begin{equation*}
0<\frac{\tau}{aC_p}=\left(\frac{\tau}{aC_p}+\frac{2\psi}{C_p}\right)-\frac{2\psi}{C_p}=\alpha+\sigma=0,
\end{equation*}
a contradiction.  Thus, we must have $e=0$.
Similarly, adding the equations of \eqref{inveq5,8} gives $g(\delta-\sigma)=0$.
If $\delta-\sigma=0$, then
\begin{equation*}
0<\frac{\tau}{dC_p}=\left(\frac{\tau}{dC_p}-\frac{2\psi}{C_p}\right)-\frac{-2\psi}{C_p}=\delta-\sigma=0,
\end{equation*}
a contradiction.  Thus, we also have $g=0$.  Thus
\begin{equation*}
A=\left[\begin{array}{cc}a& if\\ if & d\end{array}\right],\, (A+I_{2})^{-1}=\left[\begin{array}{cc}\alpha & i\sigma\\ i\sigma & \delta\end{array}\right].
\end{equation*}
Using this new information, the eight equations above reduce to
\begin{eqnarray}
1=\ta\alpha-f\sigma,& &1=\td\delta-f\sigma,\label{inv2eq1,4}\\
0=\ta\sigma+f\delta,& &0=\td\sigma+f\alpha.\label{inv2eq2,3}
\end{eqnarray}
Using \eqref{coefficentequations} in \eqref{inv2eq2,3} and clearing the $C_p$ in denominator gives
\begin{eqnarray}
0=\frac{f}{d}\tau+2(-\ta-f)\psi,& &0=\frac{f}{a}\tau+2(-\td+f)\psi.\label{lasteq1,2}
\end{eqnarray}
Subtracting the two above equations (in reverse order of their appearance) of Equation \eqref{lasteq1,2} yields
\begin{equation}
0=-\left(\frac{a-d}{ad}\right)f\tau+2\psi^2\label{whenpsi0}.
\end{equation}
We would like to know that $\psi=0$, for then Equation \eqref{whenpsi0} would show that $a=d$ and $f=0$.  Showing that $\psi=0$ is involved argument, and thus we shall prove it as a separate lemma.
\begin{lemma}
Let $1<p<\infty$ and $p\neq 2$. Suppose $A\in\mathcal{A}'^{2}$ and a critical point of $h_p$.  Then $\psi=0$.
\end{lemma}
\begin{proof}
We will proceed by contradiction, assuming $\psi\neq 0$ and showing that this implies $p=2$.  So suppose $\psi\neq 0$.  
We can clear the denominators of the two equations in \eqref{lasteq1,2} and subtract the resulting equations to get
\begin{equation*}
0
=2(a-d-(a+d)f)\psi.
\end{equation*}
Since we assumed $\psi\neq 0$, we can divide the above equation by $\psi$ and solve for $f$, yielding
\begin{equation}
f=\frac{a-d}{a+d}.\label{feq}
\end{equation}
One can use the above expression of $f$ in the definition of $\psi$ in \eqref{psiphidef} to compute
\begin{equation}
\psi
=\frac{a-d}{a+d}(\ta+\td).\label{psieq}
\end{equation}
We can use \eqref{feq} and \eqref{psieq} to rewrite \eqref{whenpsi0} as
\begin{equation}
0
=-\frac{(a-d)^2}{ad(a+d)}\tau+2\frac{(a-d)^2}{(a+d)^2}(\ta+\td)^2\label{inter1}.
\end{equation}
Since we assume $\psi\neq 0$, by \eqref{psieq} we have $a-d\neq 0$.  Thus, we can multiply each side of \eqref{inter1} by $\frac{ad(a+d)^2}{(a-d)^2}$
\begin{equation}
0=-(a+d)\tau+2ad(\ta+\td)^2\label{newtau}.
\end{equation}
Now by \eqref{tauform1} we have
\begin{equation}
\tau
= (ad+f^2-1)^2+8ad+2ad(a+d)+2a(f-1)^2+2d(f+1)^2.\label{tauform1.1}
\end{equation}
Using \eqref{feq} in \eqref{tauform1.1}, one (or a computer algebra system) can show that 
\begin{equation}
\tau
=\left(ad-1+\frac{(a-d)^2}{(a+d)^2}\right)^2+\frac{2ad}{a+d}(\ta+\td)^2.\label{tauwithf}
\end{equation}
Plugging \eqref{tauwithf} into \eqref{newtau}, one (or a computer algebra system) can show that 
\begin{equation*}
0
=-(a+d)\left(ad-1+\frac{(a-d)^2}{(a+d)^2}\right)^2.
\end{equation*}
Since $a+d>0$, we can divide by $a+d$ above, take the square root, clear the denominator, and rearrange the equality to get
\begin{equation*}
ad(a+d)^2=(a+d)^2-(a-d)^2=4ad.
\end{equation*}
Dividing each side by $ad$ (which, note, is not $0$) and taking the square root yields
\begin{equation}
a+d=2\label{a+d}.
\end{equation}
The above equation \eqref{a+d} allows us to simplify things more.  In fact, we can change \eqref{feq}, \eqref{psieq}, \eqref{newtau} into
\begin{equation}
f=\frac{a-d}{2},\, \psi=2(a-d),\, \tau=16ad \label{feq2,psieq2,newtau2}
\end{equation}
One can use \eqref{feq2,psieq2,newtau2} above to give an simple expression for $C_p$:
\begin{equation}
C_p=32+16(p-2)ad\label{Cpeq}.
\end{equation}
However, we also have the first equation in \eqref{inv2eq1,4} which, combined with \eqref{coefficentequations} and \eqref{feq2,psieq2,newtau2} says
\begin{equation}
C_p=C_p(\ta\alpha-f\sigma)=\ta\left(\frac{\tau}{a}+2\psi\right)+2f\psi
=6a^2+8ad+2d^2+12d+4a\label{Cpeq1}.
\end{equation}
Similiarly, one can rewrite the second equation in \eqref{inv2eq1,4} as
\begin{equation}
C_p=C_p(\td\delta-f\sigma)=2a^2+8ad+6d^2+12a+4d\label{Cpeq2}.
\end{equation}
We now combine equations \eqref{Cpeq}, \eqref{Cpeq1}, and \eqref{Cpeq2} to get
\begin{eqnarray*}
32+16(p-2)ad&=&C_p=\frac{1}{2}C_p(\ta\alpha-f\sigma+\td\delta-f\sigma)\\
&=&\frac{1}{2}(8a^2+16ad+8d^2+16a+16d)=32.
\end{eqnarray*}
Thus, we must have
\begin{equation}
16(p-2)ad=0.
\end{equation}
But $a>0$ and $d>0$, so the only way the above equation can hold is if $p=2$.  This contradicts our assumption that $p\neq 2$, proving the lemma.
\end{proof}
Thus $\psi=0$.  By \eqref{whenpsi0}, we know that either $a-d=0$ or $f=0$.  Since $0=\psi=(a-d)+2f$, this implies that \emph{both} $a-d=0$ and $f=0$.  Thus the critical point $A=\left[\begin{array}{cc}a& 0\\ 0 & a\end{array}\right]$.
We can rewrite the equation the first equation in \eqref{inv2eq1,4} as $1=\ta\alpha$.  Plugging in \eqref{coefficentequations} into $1=\ta\alpha$ and solving for $a$ yields $a=\frac{1}{p-1}$, proving $A=\frac{1}{p-1}I_2$.
It is elementary to verify that the matrix $\frac{1}{p-1}I_2$ is a critical point, proving that is is the unique critical point in $\mathcal{A}'^2$.
\end{proof}
We can turn this special case into a general theorem.  However, we will need an intermediate lemma to prove the theorem.
\begin{lemma}\label{diagonalize}
Let $A\in\mathcal{A}^{2}$.  If $U\in SO(2)$ (i.e. $U$ is a real orthogonal matrix with $\det(U)=1$), then $h_p(A)=h_p(U^TAU)$. 
\end{lemma}
\begin{proof}
Let $A\in\mathcal{A}^{2}$ and $U\in SO(2)$.
Since $U$ is real, $SO(2)$ is commutative, and $U,J\in SO(2)$, we have
\begin{eqnarray*}
h_p(U^TAU)&=&\frac{\det(\Re(U^TAU))}{|\det(U^T(A+I_{2})U)|^p\det(I_{2}+\Omega((U^TAU+I_2)^{-1}))}\\
&=& \frac{\det(\Re(A))}{|\det(A+I_{2})|^p\det(I_{2n}+\Omega(B))}= h_p(A),
\end{eqnarray*}
as claimed.
\end{proof}
With Lemma \ref{diagonalize}, we can prove that there is only one critical point in $\mathcal{A}^2$.
\begin{theorem}\label{criticalpointthm}
Let $1<p<\infty$ and $p\neq 2$.  The function $h_p:\mathcal{A}^2\to(0,\infty)$ defined by
\begin{equation}
h_p(A):=\frac{\det(\Re(A))}{|\det(A+I_2)|^p\det(I_{2n}+\Omega((I_2+A)^{-1}))}.
\end{equation}
has exactly one critical point, namely $A=\frac{1}{p-1}I_2$.
\end{theorem}
\begin{proof}
Let $A\in\mathcal{A}^2$ be a critical point of $h_p$ (note by Proposition \ref{lemcriticalpointb=0} at least one critical point exists).  Since $\Re(A)$ is symmetric, there exists a $U\in SO(2)$ such that $U^TAU\in\mathcal{A}'^2$.  Since the mapping $B\to UBU^T$ is a diffeomorphism of $\mathcal{A}^2$ to itself and $h_p(B)=h_p(UBU^T)$ by Lemma \ref{diagonalize}, $U^TAU$ must also be a critical point of $h_p$.  By Lemma \ref{lemcriticalpointb=0}, we must have $U^TAU=\frac{1}{p-1}I_2$, which forces $A=\frac{1}{p-1}I_2$.
\end{proof}
Here we note the value of $h_p$ at the critical point $A=\frac{1}{p-1}I_2$.  It is straightforward to calculate that
\begin{equation*}
h_p\left(\frac{1}{p-1}I_2\right)=\left(\frac{(p-1)^{p-1}}{p^p}\right)^2=\left(\frac{1}{p^{1/p}}\frac{1}{p'^{1/{p'}}}\right)^{2p}.
\end{equation*}
We will need to know the behavior of this prospective maximum value in what follows. 
\begin{lemma}\label{minlemma}
The function $j\colon(1,\infty)\to\mathbb{R}$ defined as $j(p)\equiv \frac{1}{p^{1/p}}\frac{1}{p'^{1/{p'}}}$ takes a minimum value at $p=2$ and $j(2)=\frac{1}{2}$.
\end{lemma}
\begin{proof}
This can be shown using elementary calculus.  The details are omitted.
\end{proof}
We will use the above lemma when proving that our critical point gives us a unique maximum when $p\neq 2$.
%
%
%
%
\subsection{Proving the Maximum Occurs at the Critical Point} \label{section maximum at critical point}
We have a unique critical point $\frac{1}{p-1}I_2$ for our function $h_p$, and next we want to show that this critical point gives us our maximum.  Our plan is to define a compact set $\mathcal{K}\subset\mathcal{A}^2$ such that $h_p$ takes on values strictly less than $h_p\left(\frac{1}{p-1}I_2\right)=\left(\frac{1}{p^{1/p}}\frac{1}{p'^{1/{p'}}}\right)^{2p}$ outside of and on the boundary of $\mathcal{K}$.  Thus our first job is going to be finding appropriate bounds with which we will construct $\mathcal{K}$.
 
To find these bounds, we are going to define two common operator norms on matrices.  For a vector $v=\left[\begin{array}{c} v_1\\v_2\end{array}\right]\in\R^2$, denote the $2$-norm by $|v|_2$ ($|v|_2=(|v_1|^2+|v_2|^2)^{1/2}$). For $B\in\R^{2\times 2}$, denote the operator $2$-norm as $|B|_2$ ($|B|_2=\max\{|Bv|_2:|v|_2=1\}$).  We will use the fact that 
\begin{equation}
|U^TBU|_2=|B|_2 \mbox{ for any $U\in SO(2)$ and any $B\in\R^{2\times 2}$.} \label{prop2normunitary}
\end{equation}
We also denote the maximum norm as $|B|_{\max}$ ($|B|_{\max}=\max_{i=1,2; j=1,2}|b_{ij}|$).  The norms $|\cdot|_{2}$ and $|\cdot|_{\max}$ are equivalent.  In fact,
\begin{equation}
|B|_{\max}\leq|B|_{2}\leq 2 |B|_{\max}\mbox{ for any $B\in\R^{2\times 2}$.}\label{propnormequivalence}
\end{equation}
Equations \eqref{prop2normunitary} and \eqref{propnormequivalence} are well-known facts.  See, for example, \cite{HornJohnson}.

\medskip

First, we start with a bound for the operator $|Q_\alpha|$ (whose definition can be found in Equation \eqref{eq |Qa|}) that will prove useful.  Since it is as easy to prove in general dimension as in $1$ dimension, we state it for general $n$.

\begin{lemma}\label{|Qalpha|lemma}
Let $A_r$ be a real positive definite $2n$-by-$2n$ matrix with eigenvalues $\lambda_j$ for $j=1,...,2n$.  Let $g_r(x)=e^{-(x,\frac{\alpha}{2}A_rx)}$.  Then we have
$$\frac{\||Q_\alpha|g_r\|_p^p}{\|g_r\|_p^p}=2^{np}\prod_{i=1}^{2n}\sqrt{\frac{1}{(1+\lambda_i)^{p-1}}}$$
\end{lemma}
\begin{proof}
Using \eqref{gaussianintegral}, one can calculate
\begin{eqnarray*}
\frac{\||Q_\alpha|g_r\|_p^p}{\|g_r\|_p^p}&=&2^{np}\sqrt{\frac{\det(A_r)}{\det(A_r+I_{2n})^p\det(I_{2n}-(A_r+I_{2n})^{-1})}}\\
&=&2^{np}\prod_{j=1}^{2n}\sqrt{\frac{1}{(\lambda_j+1)^{p-1}}}.
\end{eqnarray*}
\end{proof}

In the next lemma, we define our first bound $M_p^{a,d}$.
%
\begin{lemma}\label{lembounds1}
Let $1<p<\infty$.  There exists a positive real number $M_p^{a,d}$ such that for any $A\in\mathcal{A}^2$ if $|\Re(A)|_2\geq M_p^{a,d}$, then $h_p(A)<\left(\frac{1}{p^{1/p}}\frac{1}{p'^{1/{p'}}}\right)^{2p}$.
\end{lemma}
\begin{proof}
Let $A\in\mathcal{A}^2$, and define $A_r=\Re(A)$.  Let $\lambda_1$ and $\lambda_2$ be the eigenvalues of $A_r$.  Define $g(x)=e^{-(x,\frac{\alpha}{2}Ax)}$ and $g_r(x)=e^{-(x,\frac{\alpha}{2}A_rx)}$.  Note that $\|g\|_p=\|g_r\|_p$. Then using Lemma \ref{|Qalpha|lemma} we have
\begin{eqnarray*}
h_p(A)
&\leq &\left(\frac{1}{2^{np}}\frac{\||Q_\alpha| g_r\|_p^p}{\|g_r\|_p^p}\right)^2= \frac{1}{(\lambda_1+1)^{p-1}}\frac{1}{(\lambda_2+1)^{p-1}}\\
&<& 1\cdot \frac{1}{(\max(\lambda_1,\lambda_2)+1)^{p-1}}=\frac{1}{(|\Re(A_r)|_{2}+1)^{p-1}}.
\end{eqnarray*}
Since $p-1>0$, we have
$$\lim_{x\to\infty}\frac{1}{(x+1)^{p-1}}=0,$$
and thus there exists a $M_p^{a,d}$ such that for $x\geq M_p^{a,d}$ we have $\frac{1}{(x+1)^{p-1}}<\left[\frac{1}{p^{1/p}}\frac{1}{p'^{1/p'}}\right]^{2p}$.  Thus, if $|\Re(A)|_{2}\geq M_p^{a,d}$, we have
$$h_p(A)<\frac{1}{(|\Re(A')|_{\max}+1)^{p-1}}<\left[\frac{1}{p^{1/p}}\frac{1}{p'^{1/p'}}\right]^{2p},$$
as desired.
\end{proof}
For our next bound $M_p^{e,f,g}$ we will just consider matrices in $\mathcal{A}'^2$.  
%
\begin{lemma}\label{lembounds2}
Let $1<p<\infty$.  There exists a positive real number $M_p^{e,f,g}$ such that for any $A'\in\mathcal{A}'^2$, if $|\Re(A')|_{2}< M_p^{a,d}$ and $|\Im(A')|_{\max}\geq M_p^{e,f,g}$, then $h_p(A')<\left(\frac{1}{p^{1/p}}\frac{1}{p'^{1/{p'}}}\right)^{2p}$.
\end{lemma}
\begin{proof}
First we define the bound $M_p^{e,f,g}$ as 
\begin{equation}
M_p^{e,f,g}:=\left[\frac{1}{p^{1/p}}\frac{1}{p'^{1/p'}}\right]^{-2}\cdot(2(M_p^{a,d})^2+2M_p^{a,d}+3)^{1/p}.\label{Mpefg} 
\end{equation}
Let $A'\in\mathcal{A}'^2$ with $|\Re(A')|_{2}< M_p^{a,d}$ and $|\Im(A')|_{\max}\geq M_p^{e,f,g}$.  Note that since $\Re(A')$ is diagonal, we have $|\Re(A')|_{\max}=|\Re(A')|_{2}< M_p^{a,d}$ and that we can write
\begin{equation*}
h_p(A')
=\frac{ad(\Psi^2+\Phi^2)}{(\Psi^2+\Phi^2)^\frac{p}{2}\tau}
=\frac{ad}{(\Psi^2+\Phi^2)^\frac{p}{2}}\left(1+\frac{\psi^2+\phi^2}{\tau}\right).
\end{equation*}
First we bound $ad\frac{\psi^2+\phi^2}{\tau}$.  Repeatedly using Cauchy's inequality ($2|xy|\leq(x^2+y^2)$) yields
\begin{eqnarray}
\psi^2+\phi^2
&\leq& 3a^2-6ad+3d^2+6f^2+2e^2+2g^2\nonumber\\
&<&3a^2+3d^2+6f^2+2e^2+2g^2.\label{psiphibound1}
\end{eqnarray}
Thus, using the above equation, the expression of $\tau$ as \eqref{tauform2}, and the fact $a,d < M_p^{a,d}$ we have
\begin{eqnarray}
ad\frac{\psi^2+\phi^2}{\tau}&<&\frac{ad(3a^2+3d^2+6f^2+2e^2+2g^2)}{\tau}\nonumber\\
&\leq& \frac{3a^3d}{6ad}+\frac{3ad^3}{6ad}+\frac{6adf^2}{2adf^2}+\frac{2ade^2}{2de^2}+\frac{2adg^2}{2ag^2}\nonumber\\
&\leq& (M_p^{a,d})^2+2M_p^{a,d}+3.\label{efgbound1}
\end{eqnarray}
Now, using \eqref{tauform2} one (or a computer algebra system) can show
\begin{eqnarray}
\Psi^2+\Phi^2
&=&(eg-f^2)^2+a^2d^2+2ad(a+d)+6ad+2adf^2+2af^2+\nonumber\\
& & 2(a+d)f^2+2(a+d)+e^2(d^2+2d)+g^2(a^2+2a)+\nonumber\\
& & (a-d)^2+1+2f^2+e^2+g^2.\label{efgbound2}
\end{eqnarray}
We have four useful inequalities that we can deduce from \eqref{efgbound2} that we will summarize as one inequality
\begin{equation}
\Psi^2+\Phi^2\geq \max(1,f^2,e^2,g^2).\label{efgbound3}
\end{equation}
Using \eqref{efgbound1} and \eqref{efgbound3} we have
\begin{eqnarray}
h_p(A')
&<&\frac{1}{(\Psi^2+\Phi^2)^\frac{p}{2}}(2(M_p^{a,d})^2+2M_p^{a,d}+3)\nonumber\\
&\leq&\frac{1}{(\max(f^2,e^2,g^2))^\frac{p}{2}}(2(M_p^{a,d})^2+2M_p^{a,d}+3)\nonumber\\
&\leq&\frac{1}{(M_p^{e,f,g})^p}(2(M_p^{a,d})^2+2M_p^{a,d}+3)\nonumber
=\left[\frac{1}{p^{1/p}}\frac{1}{p'^{1/p'}}\right]^{2p}\nonumber,
\end{eqnarray}
proving the lemma.
\end{proof}
We have one last bound to define, a lower bound that we will call $m_p^{a,d}$.  We need a lower bound for the real parts of matrices in $\mathcal{A}^2$ since $h_p$ does not extend continuously over the set of symmetric matrices whose real part is positive semidefinite.  The issue is that for a matrix in the closure of $\mathcal{A}^2$, it is possible for the $\tau$ in the denominator of $h_p$ to vanish when either $a=0$ or $d=0$.  In fact, if we just consider matrices in the closure of $\mathcal{A}'^2$, one can check that $\tau$ vanishes in exactly two cases:
\begin{enumerate}
\item $a=0, e=0, f=-1$
\item $d=0, g=0, f=1$
\end{enumerate}
We will have to consider how $h_p(A')$ behaves when $A'\in\mathcal{A}'^2$ has entries close to one of the two cases above. As we will see, these two cases on the boundary will require us to again impose the condition that $p\neq 2$, a condition that was not needed in the previous two lemmas.
%
%
%
\begin{lemma}\label{lembounds3}
Let $1<p<\infty$ and $p\neq 2$.  There exists a positive real number $m_p^{a,d}$ such that for any $A'=\left[\begin{array}{cc}a+ie & if\\ if & d+ig\end{array}\right]\in\mathcal{A}'^2$ if $\min(a,d)\leq m_p^{a,d}$, $|\Re(A')|_{2}< M_p^{a,d}$ and $|\Im(A')|_{\max}< M_p^{e,f,g}$, then $h_p(A')<\left(\frac{1}{p^{1/p}}\frac{1}{p'^{1/{p'}}}\right)^{2p}$.
\end{lemma}
%
%
\begin{proof}
Let $A'\in\mathcal{A}'^2$ satisfy $|\Re(A')|_{2}< M_p^{a,d}$ and $|\Im(A')|_{\max}< M_p^{e,f,g}$.  We will first concern ourselves with $h_p(A')$ when $A'$ has entries close to the two cases enumerated just before the lemma where $\tau$ vanishes.  The two cases are related due to symmetry in $h_p$.  In fact, one can check that 
\begin{equation}
h_p\left(\left[\begin{array}{cc}a+ie & if\\ if & d+ig\end{array}\right]\right)=h_p\left(\left[\begin{array}{cc}d+ig & -if\\ -if & a+ie\end{array}\right]\right).\label{variableswap}
\end{equation}
We will use \eqref{variableswap} to concentrate on the $a=0, e=0, f=-1$ case. 
Define a new function $\tilde{h}_p:[0,M_p^{a,d}]^2\times[-M_p^{e,f,g},M_p^{e,f,g}]^2\times[-M_p^{e,f,g},0]\to [0,\infty)$ by
\begin{equation}\label{deftildef}
\tilde{h}_p(a,d,e,g,f):=\frac{d(\Psi^2+\Phi^2)^{1-p/2}}{ad^2+2d(a+d)+6d+2df^2+2(f-1)^2+g^2(a+2)}.
\end{equation}
First note that by \eqref{tauform2} we have
$$\tau\geq a^2d^2+2ad(a+d)+6ad+2adf^2+2a(f-1)^2+g^2(a^2+2a),$$
an inequality that is close to equality when $e$ and $f+1$ are close to 0.  Using the above one can show that $h_p(A')\leq \tilde{h}_p(a,d,e,g,f)$.
Note that $\tilde{h}_p$ is uniformly continuous on its domain.  Let $\epsilon_p$ be defined as
\begin{equation}\label{defepsilonp}
\epsilon_p:=\left[\frac{1}{p^{1/p}}\frac{1}{p'^{1/p'}}\right]^{2p}-\left(\frac{1}{2}\right)^p\left[\frac{1}{p^{1/p}}\frac{1}{p'^{1/p'}}\right]^{p}.
\end{equation}
By Lemma \ref{minlemma}, we have $\frac{1}{p^{1/p}}\frac{1}{p'^{1/p'}}>\frac{1}{2}$ for $p\neq 2$.
Thus, $\epsilon_p>0$ for $p\neq 2$.  By the uniform continuity of $\tilde{h}_p$, there exists a $\delta_p>0$ not dependent on the choice of $A'$ such that
\begin{eqnarray}
\max(a,|e|,|f+1|)<\delta_p \implies |\tilde{h}_p(a,d,e,g,f)-\tilde{h}_p(0,d,0,g,-1)|<\epsilon_p\nonumber\\
\implies h_p(A')\leq\tilde{h}_p(a,d,e,g,f)<\tilde{h}_p(0,d,0,g,-1)+\epsilon_p.\label{smallabound1}
\end{eqnarray}
We want to maximize $\tilde{h}_p(0,d,0,g,-1)$ (and justify our choice of $\epsilon_p$).  To that end,
\begin{equation*}
\tilde{h}_p(0,d,0,g,-1)
=\frac{d}{2((d+2)^2+g^2)^{p/2}}
\leq \frac{d}{2(d+2)^{p}}.
\end{equation*}
Maximizing this last expression over $[0,\infty)$, we see that the maximum occurs at $x=\frac{2}{p-1}$.  So we have
\begin{equation*}
\tilde{h}_p(0,d,0,g,-1)
\leq\frac{\frac{2}{p-1}}{2\left(\frac{2}{p-1}+2\right)^p}
=\left(\frac{1}{2}\right)^p\left(\frac{1}{p^{1/p}}\frac{1}{p'^{1/p'}}\right)^p.
\end{equation*}
Thus, using \eqref{defepsilonp} we have 
\begin{equation}
\tilde{h}_p(0,d,0,g,-1)+\epsilon_p\leq\left(\frac{1}{2}\right)^p\left(\frac{1}{p^{1/p}}\frac{1}{p'^{1/p'}}\right)^p+\epsilon_p
=\left[\frac{1}{p^{1/p}}\frac{1}{p'^{1/p'}}\right]^{2p}\label{smallabound3}.
\end{equation}
Putting \eqref{smallabound1} and \eqref{smallabound3} together we have
\begin{equation}
\max(a,|e|,|f+1|)<\delta_p \implies h_p(A')<\left[\frac{1}{p^{1/p}}\frac{1}{p'^{1/p'}}\right]^{2p}.\label{smallabound4}
\end{equation}
By the symmetry \eqref{variableswap}, the equation \eqref{smallabound4} also gives us
\begin{equation}
\max(d,|g|,|f-1|)<\delta_p \implies h_p(A')<\left[\frac{1}{p^{1/p}}\frac{1}{p'^{1/p'}}\right]^{2p}.\label{smalldbound4}
\end{equation}
We also need to consider the case that either $|e|$ or $|f+1|$ are greater than $\delta_p$. Note that by \eqref{psiphibound1} and \eqref{tauform2} and since $|\Re(A')|_{\max}< M_p^{a,d}$ and $|\Im(A')|_{\max}< M_p^{e,f,g}$ with $\max(|e|,|f+1|)\geq\delta_p$, we have
\begin{eqnarray}
h_p(a,d,e,f,g)
&\leq&\frac{ad(\Phi^2+\Psi^2)}{\tau}\leq a\left(d+\frac{d(\phi^2+\psi^2)}{\tau}\right)\nonumber\\
&\leq& a\left(d+\frac{d(\phi^2+\psi^2)}{2d(f+1)^2+2de^2}\right)\leq a\left(d+\frac{(\phi^2+\psi^2)}{2\delta_p^2}\right)\nonumber\\
&\leq& a\left(d+\frac{3a^2+3d^2+6f^2+2e^2+2g^2}{2\delta_p^2}\right)\nonumber\\
&\leq&a \left(\frac{M_p^{a,d}\delta_p^2+3(M_p^{a,d})^2+5(M_p^{e,f,g})^2}{\delta_p^2}\right)\label{ineqef+1}
\end{eqnarray}
Thus, if we set 
\begin{equation*}
\delta'_p:=\frac{\delta_p^2}{M_p^{a,d}\delta_p^2+3(M_p^{a,d})^2+5(M_p^{e,f,g})^2}\frac{1}{2}\left[\frac{1}{p^{1/p}}\frac{1}{p'^{1/p'}}\right]^{2p},
\end{equation*}
then \eqref{ineqef+1} proves the implication
\begin{eqnarray}
a\leq\delta'_p,\,\max(|e|,|f+1|)\geq\delta_p
\implies h_p(A')<\left[\frac{1}{p^{1/p}}\frac{1}{p'^{1/p'}}\right]^{2p}.  \label{smallabound5}
\end{eqnarray}
Again, by symmetry \eqref{variableswap}, we also have
\begin{eqnarray}
d\leq\delta'_p,\, \max(|g|,|f-1|)\geq\delta_p
\implies h_p(A')<\left[\frac{1}{p^{1/p}}\frac{1}{p'^{1/p'}}\right]^{2p}.  \label{smalldbound5}
\end{eqnarray}
Let $m_p^{a,d}=\min(\delta_p,\delta'_p,M_p^{a,d})$.  Note that since $\delta_p$ and $\delta_p'$ do not depend on our choice of $A'$, $m_p^{a,d}$ also does not depend on our choice of $A'$.  Then combining \eqref{smallabound4}, \eqref{smalldbound4}, \eqref{smallabound5}  and \eqref{smalldbound5} we have 
\begin{eqnarray*}
\min(a,d)\leq m_p^{a,d}\implies
h_p(A')<\left[\frac{1}{p^{1/p}}\frac{1}{p'^{1/p'}}\right]^{2p},  
\end{eqnarray*}
proving the lemma.
\end{proof}
The combination of Lemmas \ref{lembounds1}, \ref{lembounds2}, and \ref{lembounds3} give
%
%
\begin{proposition}\label{lembounds}
Let $1<p<\infty$ and $p\neq 2$.  There exists positive real numbers $m_p^{a,d},M_p^{a,d},M_p^{e,f,g}$ such that for any $A'=\left[\begin{array}{cc}a+ie & if\\ if & d+ig\end{array}\right]\in\mathcal{A}'^2$ if either $|\Re(A')|_{2}\geq M_p^{a,d}$, or $|\Im(A')|\geq M_p^{e,f,g}$, or $\min(a,d)\leq m_p^{a,d}$, then $h_p(A')<\left(\frac{1}{p^{1/p}}\frac{1}{p'^{1/{p'}}}\right)^{2p}$
\end{proposition}
%
We now have all the bounds we need to prove our final theorem.  
%
%
\begin{theorem}\label{h_pmax}
Let $1<p<\infty$ with $p\neq 2$.  Then the function $h_p(A)$ takes a unique maximum value at $A=\frac{1}{p-1}I_2$ of $\left[\frac{1}{p^{1/p}}\frac{1}{p'^{1/p'}}\right]^{2p}$.
\end{theorem}
%
%
\begin{proof}
First we will use the bounds from Proposition \ref{lembounds} to create a compact set.  Define the set $\mathcal{K}\subset\mathcal{A}^2$ as
\begin{eqnarray}
\mathcal{K}&:=&\{A\in\mathcal{A}^2: |\Im(A)|_2\leq 2M_p^{e,f,g}, \mbox{ the eigenvalues } \lambda_1,\lambda_2\nonumber\\
& & \mbox{of $\Re(A)$ satisfy $m_p^{a,d}\leq \lambda_j\leq M_p^{a,d}$ for $j=1,2$}\}.\label{defmathcalK}
\end{eqnarray}
Note that $\mathcal{K}$ is compact.    
Thus, $h_p$ takes a maximum value on $\mathcal{K}$ that occurs either at the critical point $\frac{1}{p-1}I_2$ or the boundary of $\mathcal{K}$.  We will show that if $A\in\mathcal{A}^2-\mathrm{int}(\mathcal{K})$ (here $\mathrm{int}(\mathcal{K})$ is the interior of $\mathcal{K}$), then $h_p(A)<h_p(\frac{1}{p-1}I_2)=\left[\frac{1}{p^{1/p}}\frac{1}{p'^{1/p'}}\right]^{2p}$.  This will prove that $h_p(\frac{1}{p-1}I_2)=\left[\frac{1}{p^{1/p}}\frac{1}{p'^{1/p'}}\right]^{2p}$ is indeed the maximum over all of $\mathcal{A}^2$, proving the result of the theorem.  So let $A\in\mathcal{A}^2-\mathrm{int}(\mathcal{K})$ with eigenvalues $\lambda_1$ and $\lambda_2$.  Since $A\notin \mathrm{int}(\mathcal{K})$, either
\begin{equation}\label{Aconditions1}
|\Im(A)|_2\geq 2M_p^{e,f,g},\, \min(\lambda_1,\lambda_2)\leq m_p^{a,d}, \mbox { or } \max(\lambda_1,\lambda_2)=|\Re(A)|_2\geq M_p^{a,d}.
\end{equation}
Choose $U\in SO(2)$ such that $U^T\Re(A)U$ is diagonal and let $A'=U^TAU$.  By Lemma \ref{diagonalize} $h_p(A)=h_p(A')$.
Write $A'$ as $A'=\left[\begin{array}{cc}a+ie & if\\ if & d+ig\end{array}\right]$.
Without loss of generality, $a=\lambda_1$, $d=\lambda_2$.  Also, since $U$ is real, note that $\Re(A')=U^T\Re(A)U$ and $\Im(A')=U^T\Im(A)U$.  So equations \eqref{prop2normunitary}, \eqref{propnormequivalence} 
and \eqref{Aconditions1} imply
\begin{equation*}
|\Im(A')|_{\max}\geq M_p^{e,f,g},\, \min(a,d)\leq m_p^{a,d}, \mbox { or } |\Re(A')|_{2}\geq M_p^{a,d}.
\end{equation*}
Thus, by Proposition \ref{lembounds} and Lemma \ref{diagonalize}, we must have
$$h_p(A)=h_p(A')<\left[\frac{1}{p^{1/p}}\frac{1}{p'^{1/p'}}\right]^{2p},$$
proving that $h_p(\frac{1}{p-1}I_2)=\left[\frac{1}{p^{1/p}}\frac{1}{p'^{1/p'}}\right]^{2p}$ is the maximum of $h_p$.
\end{proof}

Following Section \ref{section Optimization}, this concludes the proof that $\|P_\alpha\|_{p\to p} = \left(2\frac{1}{p^{1/p}p'^{1/p'}}\right)^n$ in general; i.e.\ we have proved Theorem \ref{main theorem 1}.

\bigskip

\noindent {\bf Acknowledgment}.  The authors wish to thank Brian Hall and Treven Wall for useful conversations that contributed to the development of this work, as well as the referee for his or her helpful comments.

\bibliography{Fock-kernel-ref}
\bibliographystyle{acm}

\end{document}